\documentclass[11pt]{article}
\usepackage{amsfonts,epsf,amsmath,mathrsfs}
\usepackage{amssymb}
\usepackage{graphicx}
\usepackage{tikz}
\usepackage[cp1250]{inputenc}
\usepackage{amssymb,amsfonts,amsmath, cite}
\usepackage{enumerate}

\usepackage[english]{babel}
\usepackage[cp1250]{inputenc}

\usepackage{algorithm2e}

\usepackage{longtable}

\usepackage{graphicx}
\usepackage{index}
\usepackage{fancyhdr}
\usepackage{lhelp}
\usepackage{optparams}
\usepackage{psfrag}
\usepackage{setspace}
\usepackage{tikz}
\usepackage{subcaption}
\usepackage{pgffor}

\definecolor{lgray}{gray}{0.75}

\usetikzlibrary{decorations.pathreplacing,automata,calc,positioning}
\usepackage{tikz}

\usetikzlibrary{decorations.pathreplacing,automata,calc,positioning}
\usetikzlibrary{arrows}

\newtheorem{theorem}{\bf Theorem}
\newtheorem{corollary}[theorem]{\bf Corollary}
\newtheorem{lemma}[theorem]{\bf Lemma}
\newtheorem{proposition}[theorem]{\bf Proposition}
\newtheorem{problem}{\bf Problem}

\newtheorem{definition}[theorem]{\bf Definition}

\newcommand{\qed}{\hfill $\Box$ \bigskip}

\newcommand{\cigA}{\chi_{ig}^{A}}
\newcommand{\cigB}{\chi_{ig}^{B}}
\newcommand{\cigAB}{\chi_{ig}^{AB}}

\begin{document}

\title{On game chromatic vertex-critical graphs}
\author{
Marko Jakovac$\,^{a,b}$
\and
Da\v sa \v Stesl$\,^{a,c}$
}

\date{\empty}
\maketitle

\vspace{-5mm}

\begin{center}
$^a$ Faculty of Natural Sciences and Mathematics, University of Maribor, Slovenia\\
Koro\v{s}ka cesta 160, 2000 Maribor, Slovenia \\
\medskip
$^b$ Institute of Mathematics, Physics and Mechanics, Ljubljana, Slovenia\\
Jadranska 19, 1000 Ljubljana, Slovenia\\
\texttt{marko.jakovac@um.si}\\
\medskip
$^c$ Faculty of Computer and Information Science, University of Ljubljana, Slovenia\\
Ve\v{c}na pot 113, 1000 Ljubljana, Slovenia \\
\texttt{Dasa.Stesl@fri.uni-lj.si}
\end{center}

\begin{abstract}
Several games that arise from graph coloring have been introduced and studied.  Let $\varphi$ denote a graph invariant that arises from such a game. If $G$ is a graph and $\varphi(G-x)\neq \varphi(G)=k$, $k \geq 1$, holds true for every vertex $x \in V(G)$, then $G$ is called a $k$-$\varphi$-game-vertex-critical graph. We study the concept of $\varphi$-game-vertex-criticality for $\varphi \in \{\chi_g, \chi_i, \cigA, \cigAB\}$, where $\chi_g$ denotes the standard game chromatic number, $\chi_i$ denotes the indicated game chromatic number and $\cigA$, $\cigAB$ denote two versions of the independence game chromatic number. Since the game chromatic number $\varphi(G-x)$ can either decrease or increase with respect to $\varphi(G)$, we distinguish between lower, upper and mixed vertex-criticality. We show that for $\varphi \in \{\chi_g, \cigA, \cigAB\}$ the difference $\varphi(G)-\varphi(G-x)$, $x \in V(G)$, can be arbitrarily large. A characterization of $2$-$\varphi$-game-vertex-critical and (connected) $3$-$\varphi$-lower-game-vertex-critical graphs for all $\varphi \in \{\chi_g, \chi_i, \cigA, \cigAB\}$ is given. It is shown that $\chi_g$-game-vertex-critical, $\cigA$-game-vertex-critical and $\cigAB$-game-vertex-critical graphs are not necessarily connected. However, it is also shown that $\chi_i$-lower-game-vertex-critical graphs are always connected.
\end{abstract}

\bigskip


%%%%%%%%%%%%%%%%%%%%%%%%%%%%%%%%%%%%%%%%%%%%%%%%%%%%%%%%%%%%%%%%%%%%%
\section{Introduction}
%%%%%%%%%%%%%%%%%%%%%%%%%%%%%%%%%%%%%%%%%%%%%%%%%%%%%%%%%%%%%%%%%%%%%

The concept of criticality has been explored for many different graph invariants. Perhaps the most important one is the concept of color-critical or chromatic-critical graphs; see for instance the book~\cite{book}. Recall that $G$ is a color-critical graph if $\chi(H)<\chi(G)$ holds true for any proper subgraph $H$ of $G$, where $\chi$ denotes the \emph{chromatic number}\footnote{The chromatic number is the minimum number of colors used in a proper coloring of a graph.}. In addition, the concept of vertex-chromatic-critical (or simply vertex-critical) graphs has aroused the interest of many researchers. Note that $G$ is a vertex-critical graph if $\chi(G-x)<\chi(G)$ holds true for every vertex $x \in V(G)$. The wide interest given to the concept of graph criticality is reflected in many different criticality concepts which have been investigated thus far, one of them being related to the (total) domination game; see~\cite{brklra-2010,dgc-2015,tdgc-2018}.

In this paper we introduce the criticality of graphs with respect to three variations of coloring games, more precisely with respect to the classical coloring game, with respect to the indicated coloring game, and lastly with respect to the independence coloring game.

Coloring games in graphs were introduced independently by Gardner~\cite{ga-81} and Bodlander~\cite{bo-1991}. The introduction of the initial version of the game has led to many investigations and development of various strategies and methods~\cite{bagr-07,faigle,kir-2012}. Several connections between the game chromatic number and the well known graph invariants were discovered~\cite{cs-2013, dizh-99, kiko-2009}. This has triggered the development of several variations of the coloring game~\cite{an-2009,bgk-08,bgj-2019,go-he-2018,kitr-01,neso-01}. A good source to review some of the basic results associated with coloring games are a survey on coloring games~\cite{tz-2015} and a dynamic survey on combinatorial games~\cite{F}.

One of the most intensively studied games is the initial version of the game, simply called the \emph{coloring game}. It is played on a simple finite graph $G$ by two players, Alice and Bob. Both players color the vertices of $G$ using the fixed set of colors $\{1,\ldots ,k\}$. The aim of Alice (the first player) is to color all vertices of $G$ while Bob (the second player) is trying to prevent this from happening. Alice starts the game and in the continuation of the game the players alternate turns. When choosing a color for an individual vertex both players must follow the rules of a proper coloring, i.e., they must color a vertex with a color from the color set $\{1,\ldots ,k\}$ which is different from the colors of its neighbors. If at some point of the game there exists an uncolored vertex, which has all colors from the color set $\{1,\ldots ,k\}$ in its neighborhood, Bob wins the game. Otherwise, if all the vertices of the graph are colored, Alice wins. The minimum number of colors $k$ for which Alice has a winning strategy on $G$ is called the \emph{game chromatic number} of $G$, and is denoted by $\chi_g(G)$.

The second coloring game considered in this paper is the game played by two players on a simple finite graph $G$ with a predefined fixed set of colors $\{1, \ldots ,k\}$. In this variation of the game both players are usually named Ann and Ben. In every round of the game Ann selects a previously uncolored vertex and Ben colors it with any of the available colors that have not been used in its neighborhood. Ann's goal is to achieve a proper coloring of the whole graph, while Ben has the opposite goal. He wants to create such a partial coloring of the vertices in the given graph, such that there exists an uncolored vertex, which has all colors from the color set $\{1, \ldots ,k\}$ in its neighborhood. The minimum number of colors $k$ for which Ann has a winning strategy on $G$, no matter how Ben plays, is called the {\em indicated chromatic number} of $G$, and denoted by $\chi_i(G)$. The described game was introduced by Grzesik under the name the {\em indicated coloring game}~\cite{gr-2012} and has since gained much attention from other authors~\cite{bmd-2021,la-2014,raj-2015,raj-2017}.

The last game investigated in this paper is the game initiated by Bre\v{s}ar and \v{S}tesl~\cite{db-2021} under the name the {\em independence coloring game}. Again, the game is played on a simple finite graph $G$ by two players, called Alice and Bob. It consists of several rounds whereby each round is played with the color of the round (round $1$ with color $1$, round $2$ with color $2$, and so on). In each round of the game players alternate their moves by coloring a previously uncolored vertex of the given graph with the color of the round. The vertices selected in the same round of the game must form an independent set\footnote{An independent set is a set of vertices in a graph such that no two vertices are adjacent.}. To be more precise, a round is completed when there is no longer a vertex that could be colored by the color of the round such that it is not adjacent to any vertex that has already been colored in the same round. The game is over when all vertices are colored. It follows that the total number of rounds and the number of colors that are used in the entire game are the same and the coloring obtained in the game is clearly a proper coloring. Alice wants to finish the game in as few rounds as possible and Bob in as many rounds as possible. The authors introduced four versions of this game depending on who starts each round. In this paper we consider the two natural ones where Alice has the first move. The first version is when every round of the game on a graph $G$ is started by Alice. This version  of the game is called the \emph{A-independence coloring game} and the minimum number of resulting rounds played on $G$ is called the {\em A-independence game chromatic number} of $G$, and denoted by $\cigA(G)$. The second version is when Alice starts the first round of the game, and each further round is started by the player who did not end the previous round. This version of the game is called the {\em AB-independence coloring game}, and the minimum number of resulting rounds played on $G$ is called the {\em AB-independence game chromatic number} of $G$, and denoted by $\cigAB(G)$. In both cases it is assumed that Alice and Bob play optimally in each of their moves.

%%%%%%%%%%%%%%%%%%%%%%%%%%%%%%%%%%%%%%%%%%%%%%%%%%%%%%%%%%%%%%%%%%%%%
\section{Vertex-criticality with respect to games}
%%%%%%%%%%%%%%%%%%%%%%%%%%%%%%%%%%%%%%%%%%%%%%%%%%%%%%%%%%%%%%%%%%%%%

A (chromatic) vertex-critical graph is a graph in which every vertex is a critical element in terms of the chromatic number of $G$. Thus, $\chi(G-x)<\chi(G)$ for every vertex $x \in V(G)$. It is easy to see that this notion is well defined since using the same coloring with $\chi(G)$ colors on $G-x$ which is used on $G$ yields a proper coloring of $G-x$. Since $\chi(G-x)$ is the minimum number of colors needed for a proper coloring of $G-x$, we have $\chi(G-x) \leq \chi(G)$. It turns out that vertex-critical graphs behave nicely, since if $\chi(G-x)<\chi(G)$ holds true for every vertex $x \in V(G)$, then the decrease can not be by more than $1$. Hence, if $G$ is a vertex-critical graph, then $\chi(G-x)=\chi(G)-1$ holds true for every vertex $x \in V(G)$~\cite{book}.

However, if we turn our attention to game colorings on graphs, then the relation between (game) chromatic invariants of critical graphs and their vertex-deleted subgraphs is more chaotic. When a coloring game is played on a (game-)vertex-critical graph $G$, then the (game) chromatic number of $G-x$ can either decrease or it can increase, which clearly depends on the vertex $x \in V(G)$ being removed. Moreover, the decrease or the increase can even be arbitrarily large. In this paper, we introduce vertex-critical graphs with respect to the three introduced games. Let $\varphi(G)$ denote a game chromatic invariant of a graph $G$. In accordance with the behaviour of some game chromatic invariants we consider three types of game-vertex-criticality. 

\begin{definition}
\label{varphi-lower}
A graph $G$ is \emph{$\varphi$-lower-game-vertex-critical} if $\varphi(G-x) < \varphi(G)$ holds true for every vertex $x\in V(G)$. Moreover, if $\varphi(G)=k$ for some positive integer $k$, then $G$ is called a $k$-$\varphi$-lower-game-vertex-critical graph.
\end{definition}

\begin{definition}
\label{varphi-upper}
A graph $G$ is \emph{$\varphi$-upper-game-vertex-critical} if $\varphi(G-x) > \varphi(G)$ holds true for every vertex $x\in V(G)$. Moreover, if $\varphi(G)=k$ for some positive integer $k$, then $G$ is called a $k$-$\varphi$-upper-game-vertex-critical graph.
\end{definition}

\begin{definition}
\label{varphi-mixed}
A graph $G$ is \emph{$\varphi$-mixed-game-vertex-critical} if $\varphi(G-x) \neq \varphi(G)$ holds true for every vertex $x\in V(G)$, and there exists a vertex $y \in V(G)$ such that $\varphi(G-y) < \varphi(G)$ and a vertex $y' \in V(G)$ such that $\varphi(G-y') > \varphi(G)$. Moreover, if $\varphi(G)=k$ for some positive integer $k$, then $G$ is called a $k$-$\varphi$-mixed-game-vertex-critical graph. 
\end{definition}

Combining all three definitions we give the final definition that unifies the notion of lower-, upper- and mixed-game-vertex-critical graphs.

\begin{definition}
\label{varphi-all}
A graph $G$ is \emph{$\varphi$-game-vertex-critical} if $\varphi(G-x) \neq \varphi(G)$ holds true for every vertex $x \in V(G)$. Moreover, if $\varphi(G)=k$ for some positive integer $k$, then $G$ is called a $k$-$\varphi$-game-vertex-critical graph.
\end{definition}

In other words, Definition~\ref{varphi-all} says that a graph $G$ is $\varphi$-game-vertex-critical if it is a $\varphi$-lower-game-vertex-critical, or a $\varphi$-upper-game-vertex-critical, or a $\varphi$-mixed-game-vertex-critical graph.

Our goal in this paper is to study the concept of $\varphi$-game-vertex-criticality for all $\varphi \in \{\chi_g, \chi_i, \cigA, \cigAB\}$. The core of the paper is divided into three sections. In the following section, i.e.\ Section 3, we consider coloring game-vertex-critical graphs, in Section 4 we consider indicated coloring game-vertex-critical graphs, and finally in Section 5 we consider the two versions of independence coloring game-vertex-critical graphs. We show that the difference $\varphi(G)-\varphi(G-x)$, where $x \in V(G)$ and $\varphi \in \{\chi_g, \cigA, \cigAB\}$, can be arbitrarily large in the positive and negative sense. Finally, we characterize $2$-$\varphi$-game-vertex-critical and (connected) $3$-$\varphi$-lower-game-vertex-critical graphs for all $\varphi \in \{\chi_g, \chi_i, \cigA, \cigAB\}$. We give some results that are unique to every game chromatic invariant considered in the paper. For instance, we give a nice property of $4$-$\chi_i$-lower-game-vertex-critical graphs with respect to their minimum degree. Moreover, we show that $\chi_g$-game-vertex-critical graphs, $\cigA$-game-vertex-critical graphs and $\cigAB$-game-vertex-critical graphs are not necessarily connected, which is not true for $\chi_i$-lower-game-vertex-critical graphs.

\subsection*{Notation}

All graphs considered in this paper are simple, undirected and finite. For a graph $G$ we call $|V(G)|$ the \emph{order} of graph $G$ and $|E(G)|$ the \emph{size} of graph $G$. The notations $\delta(G)$, $\Delta(G)$, $\chi(G)$ stand for the \emph{minimum vertex degree}, the \emph{maximum vertex degree} and the \emph{chromatic number} of $G$, respectively. The \emph{degree} of a vertex $v \in V(G)$ in a graph $G$ is denoted by $d_G(v)$, and the set of its neighbors is denoted by $N_G(v)$ and called the \emph{open neighborhood} of $v$. Moreover, we define $N_G[v]=N_G(v) \cup \{v\}$ as the \emph{closed neighborhood} of $v$. If $d_G(v)=1$, then $v$ is called a \emph{pendant vertex} of $G$. With $d_G(u,v)$ we denote the \emph{distance} between vertices $u$ and $v$, i.e.\ the number of edges on a shortest path between $u$ and $v$. We will simply write $d(v)$, $N(v)$, $N[v]$ and $d(u,v)$ if $G$ is the only graph considered.

Notations $P_n$, $C_n$ and $K_n$ stand for the path, the cycle, and the complete graph of order $n$, respectively. We denote by $K_{m,n}$ the complete bipartite graph of order $m+n$. Moreover, we call $K_{1,n}$ the star graph of order $n+1$. Let $\mathcal{M}$ be a perfect matching\footnote{A perfect matching of a graph is an independent set of edges in which every vertex of the graph is incident to exactly one edge of the matching.} of $K_{n,n}$. Then we denote by $K_{n,n}-\mathcal{M}$ the complete bipartite graph of order $2n$ without the perfect matching $\mathcal{M}$. If $x \in V(G)$, then $G-x$ denotes the graph obtained from $G$ by removing the vertex $x$ and all of the edges incident with $x$. Finally, suppose that $G$ is a graph of order $n$ and $u$ an extra vertex that is adjacent to all vertices of $G$. We denote the obtained graph of order $n+1$ by $G^u$, where $u \in V(G^u)$ is a \emph{universal vertex} of $G^u$. If $G$ already has at least one universal vertex, then $G^u$ has at least two universal vertices.

%%%%%%%%%%%%%%%%%%%%%%%%%%%%%%%%%%%%%%%%%%%%%%%%%%%%%%%%%%%%%%%%%%%%%
\section{The coloring game-vertex-critical graphs}
%%%%%%%%%%%%%%%%%%%%%%%%%%%%%%%%%%%%%%%%%%%%%%%%%%%%%%%%%%%%%%%%%%%%%

It is well known that vertex-critical graphs (with respect to the usual chromatic number) are connected~\cite{book}. Unfortunately, that is not always true for $\chi_g$-game-vertex-critical graphs.

\begin{proposition}
\label{game-disconnected}
For every $k \geq 4$ there exists a $k$-$\chi_g$-game-vertex-critical, disconnected graph.
\end{proposition}

\begin{proof}
Let $k \geq 4$ and let $G$ be two disjoint copies of the graph $K_{k,k}-\mathcal{M}$, where $K_{k,k}$ denotes the complete bipartite graph and $\mathcal{M}$ is a perfect matching in $K_{k,k}$. Clearly, $\chi_g(G)=k$ (whichever vertex Alice colors, Ben responds by coloring the opposite vertex along the missing edge of $\mathcal{M}$). Now, let us remove an arbitrary vertex $x \in V(G)$ and denote with $y \in V(G)$ the vertex that lies on the opposite side of the vertex $x$ along the missing edge of $\mathcal{M}$. Suppose that Alice and Bob play a coloring game on the graph $G-x$ with the color set $\{1,2,3\}$. Alice has the first move in which she colors with color $1$ the vertex $y$. Then the vertices of the partite set that contained $x$ can no longer receive color $1$. Note that Bob does not want to be the first to play in the second copy of $K_{k,k}-\mathcal{M}$, because then Alice would color (with a different color than Bob) the vertex which lies on the opposite side along the missing edge of $\mathcal{M}$, thus ensuring that only two colors would be needed to color the second copy of $K_{k,k}-\mathcal{M}$. Therefore, Bob's optimal next move is to color with color $2$ a vertex that lies in the same partite set as the vertex $y$. However, Alice responds by coloring with color $3$ the vertex that lies on the opposite side of Bob's last colored vertex along the missing edge of $\mathcal{M}$. This ensures that three colors will suffice to color the first copy of $K_{k,k}-\mathcal{M}$. Since Bob has the next move and even number of vertices of the first copy remain uncolored, Bob will automatically have the first move in the second copy of $K_{k,k}-\mathcal{M}$. Hence, Alice wins the game using three colors and $\chi_g(G-x) \leq 3$. Since $G-x$ contains $C_3$ as a subgraph, $\chi_g(G-x)=3$. Because this holds for every vertex $x \in V(G)$, $G$ is a $k$-$\chi_g$-game-vertex-critical, disconnected graph.
\qed
\end{proof}

To be more precise, the graphs considered in the proof of Proposition~\ref{game-disconnected} are $\chi_g$-lower-game-vertex-critical graphs. The proof also shows that the difference between the game chromatic number of such graphs and the game chromatic number of its vertex-deleted subgraphs can be arbitrarily large.

\begin{proposition}
\label{game-lower}
Let $n \geq 1$ be a positive integer. There exists a $\chi_g$-lower-game-vertex-critical graph $G$, such that $\chi_g(G)-\chi_g(G-x)=n$ for every vertex $x \in V(G)$.
\end{proposition}

\begin{proof}
Let  $n \geq 1$ be a positive integer and $G=K_{n+3,n+3}-\mathcal{M}$, where $\mathcal{M}$ is a perfect matching. With the same reasoning as in the proof of Proposition~\ref{game-disconnected} we conclude that $\chi_g(G)=n+3$ and $\chi_g(G-x)=3$ for every $x\in V(G)$. Thus,
$$\chi_g(G)-\chi_g(G-x)=(n+3)-3=n$$
for every vertex $x\in V(G)$.
\qed
\end{proof}

Further we give an infinity family of $\chi_g$-upper-game-vertex-critical graphs. We again show that the difference between the game chromatic number of such graphs and the game chromatic number of its vertex-deleted subgraphs can be arbitrarily large in the negative sense.

\begin{proposition}
Let $n \geq 1$ be a positive integer. There exists an $\chi_g$-upper-coloring-game-vertex-critical graph $G$, such that
$$\left\{\chi_g(G)-\chi_g(G-x) \, | \, x \in V(G)\right\}=\{-1,-n\}.$$
\end{proposition}

\begin{proof}
Let  $n \geq 1$ be a positive integer. Suppose that $G=(K_{n+3,n+3}-\mathcal{M})^u$, where $\mathcal{M}$ is a perfect matching and $u$ the universal vertex of $G$. If Alice colors the vertex $u$ in her first move, then Bob has to take all of his moves in $K_{n+3,n+3}-\mathcal{M}$ and Alice always responds with a different color on the vertex that lies on the opposite side along the missing edge of $\mathcal{M}$. Clearly, Alice wins the game on $G$ using three colors, and thus $\chi_g(G) \leq 3$. Since $G$ contains $C_3$ as a subgraph, $\chi_g(G)=3$. If we remove $u$ from $G$, we obtain the graph $K_{n+3,n+3}-\mathcal{M}$ for which $\chi_g(K_{n+3,n+3}-\mathcal{M})=n+3$ (see Proposition~\ref{game-lower}). This already proves that
$$\chi_g(G)-\chi_g(G-u)=3-(n+3)=-n.$$

We still need to show that the removal of all other vertices of $G$ also increases the game chromatic number of the vertex deleted subgraph. Let $x \in V(G)$ be an arbitrary vertex different from $u$, and let $y$ be the vertex that lies on the opposite site of vertex $x$ along the missing edge of $\mathcal{M}$. We show that $\chi_g(G-x)=4$. First we prove that Alice does not have a winning strategy in a coloring game on $G-x$ with three colors.

Suppose that Alice plays first on the vertex that is different from $u$ and different from the vertex $y$. Then Bob responds by coloring the vertex that lies on the opposite side of the vertex that Alice colored along the missing edge of $\mathcal{M}$, and he uses the same color that Alice used. With such a move he causes that no other vertex of the graph $G-x$ can receive this color in the continuation of the game. To color the remaining vertices of the graph $G-x$, Alice and Bob need at least three more colors, one for the vertex $u$ and two for the vertices of $(K_{n+3,n+3}-\mathcal{M})-x$. Hence, at least four colors are needed to complete the game on $G-x$.

Now suppose that Alice plays first either on the vertex $u$ or on the vertex $y$. If Alice colors the vertex $u$, then Bob responds by coloring the vertex $y$, or if Alice colors the vertex $y$, Bob responds with the vertex $u$. In either case, Bob has to use a new color. Further, Alice can play in her second move with a third color on an arbitrary vertex which lies in the same partite set that contained vertex $x$. In this case, Bob responds by coloring the vertex lying on the opposite side along the missing edge of $\mathcal{M}$ of the vertex Alice just colored. He uses the same color as Alice. With such a move he ensures that at least one more color will be needed to finish the game on $G-x$. Otherwise, Alice can make her second move in the partite set that contains vertex $y$. She can use the same color that was used on the vertex $y$, or a third color. If Ann uses the same color, then Bob responds by coloring a vertex in the same partite set with the third color. In either case, two different colors will be used in the partite set that contains vertex $y$. Thus, to finish the game on $G-x$ at least one new color is required. Hence, $\chi_g(G-x)\geq 4$.

To prove $\chi_g(G-x)\leq 4$, suppose that Alice continues the game considered in the latter case. She colors with the fourth color the vertex lying opposite to the vertex along the missing edge of $\mathcal{M}$ which was last colored by Bob. In this way, she wins the game on $G-x$ using four colors. Indeed, all vertices lying in the partite set that contains the vertex $y$ can now receive colors $2$ or $3$, and all other vertices can receive color $4$.
\qed
\end{proof}

If $G$ is a connected graph and $\chi_g(G)=2$, then $G$ is clearly bipartite, and if both partite sets contain at least two vertices, than no matter which vertex Alice colors in her first move, Bob responds by coloring a vertex from the same partite set with a different color, and hence at least three colors are needed to finish the game on $G$, which is a contradiction. Therefore, at least one partite set of $G$ must contain exactly one vertex, which means that $G$ must be a star graph $K_{1,n}$ for some $n \geq 1$, and it is easy to see that $\chi_g(K_{1,n})=2$. Hence, we have just characterized all connected graphs $G$ for which $\chi_g(G)=2$. Considering the rules of the coloring game, if $G$ is disconnected, then it can only be a disjoint union of stars and possibly some isolated vertices. From this fact the following proposition immediately follows.

\begin{proposition}
\label{2-game-critical}
Graph $G$ is a $2$-$\chi_g$-game-vertex-critical graph if and only if $G=K_2$.
\end{proposition}

\begin{proof}
It is easy to see that $\chi_g(K_2)=2$ and $\chi_g(K_2-x)=\chi_g(K_1)=1$ for every vertex $x \in V(G)$.

Now suppose that $\chi_g(G)=2$. If $G$ is a disjoint union of stars and possibly some isolated vertices, then removing any vertex from any of those stars or isolated vertices still yields a bipartite graph with the game chromatic number $2$. Thus, $G$ is necessarily a star graph, i.e.\ $K_{1,n}$, $n \geq 1$. Suppose that $n \geq 2$, and let $x$ and $y$ be two distinct pendant vertices of $K_{1,n}$. By removing either one of the vertices $x$ or $y$ from $K_{1,n}$ we obtain the star $K_{1,n-1}$ that still requires two colors in a coloring game. Thus, $n=1$ and $G=K_{1,1}=K_2$.
\qed
\end{proof}

Proposition~\ref{2-game-critical} shows that the notion of $2$-$\chi_g$-game-vertex-criticality is equivalent to the notion of $2$-$\chi_g$-lower-game-vertex-criticality. We also see that the only $2$-$\chi_g$-(lower)-game-vertex-critical graph is connected. But since Proposition~\ref{game-disconnected} shows that this is not always the case in a coloring game, the problem clearly becomes more involved for $k$-$\chi_g$-game-vertex-critical graphs when $k$ is large. This problem already seems challenging for $k=3$. Therefore, we characterize only connected $3$-$\chi_g$-lower-game-vertex-critical graphs.

\begin{theorem}
Let $G$ be a connected graph. Then $G$ is a $3$-$\chi_g$-lower-game-vertex-critical graph if and only if $G$ is $P_4$, $C_3$, or $C_4$.
\end{theorem}

\begin{proof}
It is again easy to see that $\chi_g(G)=3$ and $\chi_g(G-x)=2$ for every vertex $x \in V(G)$ when $G$ is either $P_4$ or $C_3$ or $C_4$. Let $G$ be a $3$-$\chi_g$-lower-game-vertex-critical, connected graph and let us denote $G'=G-x$, where $x \in V(G)$ is an arbitrary vertex.

First suppose that $\chi_g(G')=2$. Then $G'$ is a star graph or a disjoint union of stars and possibly some isolated vertices. If $G'$ is a disjoint union of stars and possibly some isolated vertices, then $x$ has at least one neighbor in every star component of $G'$ and is adjacent to every isolated vertex, because $G$ is a connected graph. Since $G'$ is a disjoint union of stars and possibly some isolated vertex, $G$ is either $P_4$ itself, which is one of the $3$-$\chi_g$-lower-game-vertex-critical graphs, or it contains the graph $P_4$ as a proper subgraph. In the latter case also $G''=G-w$ contains the graph $P_4$ as a subgraph, where $w$ denotes a pendant vertex in $G$. Suppose that Alice and Bob play a coloring game on $G''$ and Alice makes her first move on any vertex $y \in V(G'')$ which clearly lies on a path $P_4$ in $G''$. Then Bob responds with the vertex $z$ on this path that has a common neighbor with $y$, and he uses a different color than Alice. Sometime during the game between Alice and Bob one of them will have to color the common neighbor of vertices $y$ and $z$ with the third color, and hence $\chi_g(G'') \geq 3$, which means that $G$ is not a $3$-$\chi_g$-lower-game-vertex-critical graph. The other possibility is that $G'$ is a star graph, i.e.\ $K_{1,n}$ for some $n \geq 1$. If $n=1$, then the only two possibilities are either $G=P_3$ or $G=C_3$. Clearly, $C_3$ is the only $3$-$\chi_g$-lower-game-vertex-critical graph in this case. Henceforth we may assume that $n \geq 2$. Let $y$ be the central vertex and $z_1, \ldots, z_n$ the pendant vertices of $G'$.

Suppose that $x$ is adjacent to $y$ in $G$. The vertex $x$ must be adjacent to at least one of the pendant vertices of $G'$, since for otherwise $G$ would also be a star graph and hence $\chi_g(G)=2$, which is a contradiction. If there exists a pendant vertex to which $x$ is not adjacent in $G$, say $z_1$, then $G-z_1$ contains $C_3$ as a subgraph and $\chi_g(G-z_1) \geq 3$, which is again a contradiction since $G$ is a $3$-$\chi_g$-lower-game-vertex-critical graph. Therefore, $x$ must be adjacent to all pendant vertices $z_1, \ldots, z_n$, $n \geq 2$. But in this case, $G$ contains the complete graph $K_4$ as a subgraph and hence $\chi_g(G) \geq 4$, which is not possible. 

Finally, suppose that $x$ is not adjacent to $y$ in $G$, but it must be adjacent to at least one of the pendant vertices $z_1, \ldots, z_n$, $n \geq 2$, since $G$ is connected. If $n=2$, then the only two possible cases are either $G=P_4$ or $G=C_4$, which are both $3$-$\chi_g$-lower-game-vertex-critical graphs. Therefore, assume that $n \geq 3$. Removing any vertex $z_1, \ldots, z_n$ from $G$ in such a way that the remaining graph stays connected yields a graph that contains the path $P_4$ as a subgraph. With the same reasoning as above we conclude that the game chromatic number of such a graph is at least $3$, and hence $G$ is not a $3$-$\chi_g$-lower-game-vertex-critical graph.

The case for $\chi_g(G')=1$ is trivial, since $G'$ is a disjoint union of isolated vertices. Then $G$ must be a star graph for which $\chi_g(G)=2$, which is not possible. This completes the proof. 
\end{proof}

%%%%%%%%%%%%%%%%%%%%%%%%%%%%%%%%%%%%%%%%%%%%%%%%%%%%%%%%%%%%%%%%%%%%%
\section{The indicated coloring game-vertex-critical graphs}
%%%%%%%%%%%%%%%%%%%%%%%%%%%%%%%%%%%%%%%%%%%%%%%%%%%%%%%%%%%%%%%%%%%%%

It turns out that the indicated coloring game behaves nicer than the classical coloring game, at least for the case of lower-criticality, since all $\chi_i$-lower-game-vertex-critical graph are connected. To prove this we need the following lemma.

\begin{lemma}
\label{max-components}
Let $G_1, \ldots ,G_n$ be connected components of a graph $G$. Then
$$\chi_i(G)=\max\{\chi_i(G_k) \, | \, k=1, \ldots ,n\}.$$
\end{lemma}

\begin{proof}
A strategy of Ann which produces a coloring of $G$ with $\chi_i(G)$ colors can have Ann jumping back and forth between the components when selecting vertices. But since Ann only selects a vertex in each round of the indicated coloring game and Ben colors it with any available color, the moves in the game played on the disconnected graph $G$ can always be rearranged in such a way that Ann first selects all vertices inside one component, then moves to another component and repeats this procedure until she selects all vertices in every component. When she is selecting vertices in a component she does this in the same order as they were selected in this component in the original strategy. This yields $\chi_i(G_k) \leq \chi_i(G)$ for every $k \in \{1, \ldots ,n\}$. However, if $\chi_i(G_k) < \chi_i(G)$ for every $k \in \{1, \ldots ,n\}$, then according to the rearranged coloring strategy, $G$ can be colored with less than $\chi_i(G)$ colors, which is not possible. Therefore, there exists a component $G_j$ such that $\chi_i(G_j)=\chi_i(G)$, and hence $\chi_i(G)=\max\{\chi_i(G_k) \, | \, k=1, \ldots ,n\}$.
\qed
\end{proof}

\begin{proposition}
\label{lower-connected}
If $G$ is a $\chi_i$-lower-game-vertex-critical graph, then $G$ is connected.
\end{proposition}

\begin{proof}
Let $G$ be a $\chi_i$-lower-game-vertex-critical, disconnected graph and let $G_1, \ldots ,G_n$ be its connected components. Note that, since $G$ is not connected, $n\geq 2$. 

By Lemma~\ref{max-components} we have $\chi_i(G)=\max\{\chi_i(G_k) \, | \, k=1, \ldots ,n\}$ and there exists a component $G_j$ such that $\chi_i(G_j)=\chi_i(G)$. Since $n \geq 2$, we have $V(G-G_j) \neq \emptyset$. Let $x \in V(G-G_j)$, which means that $x \in V(G_\ell)$, $\ell \neq j$. If the removal of a vertex $x$ lowers $\chi_i(G_\ell)$ or if $\chi_i(G_\ell)$ stays the same, then $\chi_i(G-x)=\chi_i(G_j)=\chi_i(G)$, but if the removal of vertex $x$ increases $\chi_i(G_\ell)$, then $\chi_i(G-x) \geq \chi_i(G_j)=\chi_i(G)$. In all cases we get a contradiction to the assumption that $G$ is a $\chi_i$-lower-game-vertex-critical graph. We conclude that $G$ must be connected.
\qed
\end{proof}

Next we characterize the $2$-$\chi_i$-game-vertex-critical graphs. Similarly as for the $2$-$\chi_g$-game-vertex-criticality this shows that the notion of $2$-$\chi_i$-game-vertex-criticality is equivalent to $2$-$\chi_i$-lower-game-vertex-criticality. Recall that the indicated game chromatic number of a connected graph $G$ equals $2$ if and only if $G$ is a bipartite graph~\cite{gr-2012}. (If $G$ is disconnected, then it is a union of bipartite graphs and isolated vertices.)

\begin{proposition}
\label{2-indicated-critical}
Graph $G$ is a $2$-$\chi_i$-game-vertex-critical graph if and only if $G=K_2$. 
\end{proposition}

\begin{proof}
Obviously, $K_2$ is a $2$-$\chi_i$-game-vertex-critical graph since $\chi_i(K_2)=2$ and $\chi_i(K_2-x)=\chi_i(K_1)=1$ for any $x\in K_2$.  Further, let $G$ be an arbitrary graph, $G\neq K_2$, for which $\chi_i(G)=2$. Then, $G$ is a union of bipartite graphs and isolated vertices, and has at least three vertices. Let $x \in V(G)$. Clearly, also $G-x$ is a union of bipartite graphs and isolated vertices, or $G-x$ is a union of isolated vertices. In the first case $\chi_i(G-x)=2$ and in the latter case $\chi_i(G-x)=1$. If $\chi_i(G-x)=2$, then we immediately get a contradiction. Therefore we assume that $\chi_i(G-x)=1$. Since $G-x$ is a union of isolated vertices, and $x$ is in $G$ adjacent to some of those vertices, $K_2$ must be a proper subgraph of $G$ ($G$ has at least $3$ vertices). Removing a vertex $y \in V(G-K_2)$, yields the graph $G-y$ with $\chi_i(G-y)=2$. It follows that $G$ is not a $\chi_i$-game-vertex-critical graph. 
\qed
\end{proof}

We continue our study with $3$-$\chi_i$-game-vertex-critical graphs. Since we know that $\chi_i$-lower-game-vertex-critical graphs are connected we focus on them.

\begin{theorem}
Graph $G$ is a $3$-$\chi_i$-lower-game-vertex-critical graph if and only if $G$ is an odd cycle.
\end{theorem}

\begin{proof}
Let $G$ be an odd cycle. It is easy to see that $\chi_i(G)=3$, and $\chi_i(G-x)=2$ for every $x \in V(G)$. Thus, every odd cycle is a $3$-$\chi_i$-lower-game-vertex-critical graph.

Further, let $G$ be a $3$-$\chi_i$-lower-game-vertex-critical graph. This means that $\chi_i(G)=3$ and $\chi_i(G-x) \leq 2$ for any $x\in V(G)$. By Proposition~\ref{lower-connected}, $G$ is connected. Since for every bipartite graph the indicated game chromatic number is $2$, $G$ must contain an odd cycle, and let $C$ be the smallest odd cycle in $G$. If $G \neq C$, then there exists a vertex $x \in V(G-C)$. Clearly, $\chi_i(G-x) \geq 3$ since $C$ cannot be properly colored with two colors. In this case $G$ is not a $3$-$\chi_i$-lower-game-vertex-critical graph. The other possibility is that $G=C$ what we wanted to prove.
\qed
\end{proof}

Further we give a nice property of $k$-$\chi_i$-lower-game-vertex-critical graphs which might help characterizing them for $k \geq 4$.

\begin{proposition}
\label{indicated-property}
Let $G$ be a $\chi_i$-lower-game-vertex-critical graph. Then $d(x)\geq \chi_i(G-x)$ for every $x\in V(G)$.
\end{proposition}

\begin{proof}
We prove the contrapositive statement of our proposition. Assume that there exists a vertex $x\in V(G)$ such that $d(x)<\chi_i(G-x)$, and suppose that Ann and Ben play an indicated coloring game on $G$ using $\chi_i(G-x)$ colors. Clearly Ann has a winning strategy on $G-x$ with $\chi_i(G-x)$ colors. She selects the vertices of $G$ in the same order according to her winning strategy in $G-x$, while avoiding the vertex $x$. In the end of this process all vertices except vertex $x$ are colored with $\chi_i(G-x)$ colors. Finally, Ann selects vertex $x$. Since $d(x)<\chi_i(G-x)$ there exists at least one color from the color set $\{1,2,\ldots ,\chi_i(G-x)\}$ which is legal for $x$ and Ben has to color it with it. In this way, Ann wins the indicated coloring game on $G$ using $\chi_i(G-x)$ colors. Thus, $\chi_i(G)\leq \chi_i(G-x)$, and hence $G$ is not a $\chi_i$-lower-game-vertex-critical graph. 
\qed
\end{proof}

Even though the problem of finding all $4$-$\chi_i$-lower-game-vertex-critical graphs is considerably more challenging, we can show that the removal of an arbitrary vertex from a $4$-$\chi_i$-lower-game-vertex-critical graph lowers its indicated chromatic number by at most one.

\begin{theorem}
\label{4-indicated-lower}
Let $G$ be a $4$-$\chi_i$-lower-game-vertex-critical graph. Then $\chi_i(G-x)=3$ for every vertex $x \in V(G)$.
\end{theorem}

\begin{proof}
Let $G$ be a $4$-$\chi_i$-lower-game-vertex-critical graph. This means that $\chi_i(G-x) \leq 3$ for any $x\in V(G)$. To prove that $\chi_i(G-x)=3$ for any $x\in V(G)$ suppose on the contrary that there exists a vertex $x\in V(G)$ such that $\chi_i(G-x) \leq 2$.

First suppose that $\chi_i(G-x)=2$. Then $G-x$ is a union of (connected) bipartite graphs and isolated vertices. Now assume that Ann and Ben play an indicated coloring game on $G$ using three colors. We show that Ann has a winning strategy on $G$. We chose an arbitrary bipartite subgraph of $G-x$ and denote it by $B$. Let $G_1$ and $G_2$ be both partite sets of $B$. In the first move Ann selects vertex $x$ and without loss of generality suppose that Ben colors it with color $1$. Since $G$ is connected by Proposition~\ref{lower-connected}, $x$ must have a neighbor in $B$, say $y \in G_1$. Ann selects $y$ and Ben colors it either with color $2$ or color $3$, say $2$. In the next move, Ann selects all neighbors of $y$. Ben colors them either with color $1$ or color $3$. In this step, Ann looks at the vertices in $G_2$ which received color $3$ and selects all of their neighbors, which Ben colors either with color $1$ or color $2$. Then she selects neighbors of those vertices in $G_1$ that newly received color $2$. She repeats this procedure by alternatively selecting the neighbors of vertices in $G_1$ which have color $2$ and the neighbors of vertices in $G_2$ which have color $3$ until there are no uncolored neighbors of vertices in $G_1$ and $G_2$ colored with colors $2$ or $3$, respectively. Clearly, all vertices selected by Ann during this process could be colored by Ben, as no vertex in $G_1$ received color $3$ and no vertex in $G_2$ received color $2$. This process stops when all neighbors of vertices with colors $2$ and $3$ are colored. Note that in the last step, the uncolored neighbors of vertices with colors $2$ and $3$ received color $1$, since for otherwise the process would still continue. Hence, the vertices in $B$ that were until this step colored with colors $2$ and $3$ no longer play any role in the remainder of this game. After the first iteration, Ann can select a new neighbor of $x$ in $B$ and repeats this procedure from the beginning. It could happen that there no longer exists an uncolored neighbor of $x$ in $B$, and $B$ is still not completely colored. According to the procedure described above, there must exist an uncolored neighbor of a vertex in $B$ which received color $1$. Ann selects this vertex in her next move and Ben colors it either with with color $2$ or $3$. From this point on Ann's strategy is exactly the same as above -- she alternately selects neighbors of vertices in $G_i$, $i \in \{1,2\}$, which have color $2$ and the neighbors of vertices in $G_j$, $j \neq i$, which have color $3$ until there are no more uncolored neighbors of vertices colored with color $2$ or $3$. If $B$ is still not colored Ann finds another vertex in $B$ colored with $1$ which still has some uncolored neighbors and repeats the process. Since $B$ is connected, Ann will be able to reach all vertices of $B$ with this procedure, and $B$ will be colored with three colors. Ann can now move the game to another bipartite subgraph of $G-x$. Note that $x$ is in $G$ adjacent to all bipartite subgraphs of $G-x$, since $G$ is connected (Proposition~\ref{lower-connected}). Ann can clearly color all bipartite subgraphs of $G-x$ with three colors. What remains are the isolated vertices of $G-x$. In $G$ those vertices have degree $1$, and can easily be colored with three colors by Ben no matter in what order Ann selects them. Thus, Ann wins the game on $G$ using three colors, which contradicts the assumption $\chi_i(G)=4$. 

Now suppose that $\chi_i(G-x)=1$. In this case, $G-x$ is a union of isolated vertices. Since by Proposition~\ref{lower-connected} graph $G$ is connected, $x$ must be adjacent to all of those vertices in $G$. Hence, $G=K_{1,n}$ for some $n \geq 1$. However, $\chi_i(G)=\chi_i(K_{1,n})=2$, which is again a contradiction to $\chi_i(G)=4$.
\qed
\end{proof}

Combining Theorem~\ref{4-indicated-lower} with Proposition~\ref{indicated-property} for a $4$-$\chi_i$-lower-game-vertex-critical graph $G$ we get $d(x) \geq \chi_i(G-x)=3$ for every vertex $x\in V(G)$, which means that $4$-$\chi_i$-lower-game-vertex-critical graphs $G$ have neither pendant vertices nor vertices of degree $2$.

\begin{corollary}
If $G$ is a $4$-$\chi_i$-lower-game-vertex-critical graph, then $\delta(G) \geq 3$.
\end{corollary}

%%%%%%%%%%%%%%%%%%%%%%%%%%%%%%%%%%%%%%%%%%%%%%%%%%%%%%%%%%%%%%%%%%%%%
\section{The independence coloring game-vertex-critical graphs}
%%%%%%%%%%%%%%%%%%%%%%%%%%%%%%%%%%%%%%%%%%%%%%%%%%%%%%%%%%%%%%%%%%%%%

Similarly as in the case of $\chi_g$-game-vertex-critical graphs, we can show that a graph does not need to be connected in order to be a $\cigA$-game-vertex-critical or a $\cigAB$-game-vertex-critical graph.

\begin{proposition}
\label{independenceA,AB-disconnected}
There exists a $3$-$\cigA$-game-vertex-critical and a $3$-$\cigAB$-game-vertex-critical, disconnected graph.
\end{proposition}

\begin{proof}
Take, for instance the graph $G$, which is the disjoint union of graphs $C_6$ and $P_6$ and suppose that Alice and Bob play an A-independence or an AB-independence coloring game on $G$. Clearly, three colors are needed to finish the game on $G$. Whichever vertex Alice colors with color $1$, Bob responds by coloring with color $1$ a vertex at distance three to Alice's choice. With such a move he ensures that two more colors will be needed to complete the game. Obviously, three colors are also enough to complete the game. Therefore, $\cigA(G)=\cigAB(G)=3$. To see that the graph $G$ is indeed an A-independence and AB-independence coloring game-vertex-critical graph, let us show that $\cigA(G-x)=\cigAB(G-x)=2$ for an arbitrary vertex $x\in V(G)$. We consider two cases.

If $x\in V(C_6)$, then the graph $G-x$ is the disjoint union of paths $P_5=v_1v_2v_3v_4v_5$ and $P_6=w_1w_2w_3w_4w_5w_6$. In this case, two colors are needed to finish an A-independence and AB-independence coloring game on $G-x$ if Alice starts the game on vertex $v_3$. The moves on $P_5$ with color $1$ are now fixed. Namely, vertices $v_1$ and $v_5$ will be colored with color $1$. Alice can also ensure that Bob will be the first to play with color $1$ on $P_6$. If Bob colors $w_1$ (or $w_6$), Alice responds by coloring $w_5$ (or $w_2$), if Bob colors $w_2$ (or $w_5$), Alice responds by coloring $w_4$ (or $w_3$), and if Bob colors $w_3$ (or $w_4$), Alice responds by coloring $w_5$ (or $w_2$). When they finish using color $1$, the remaining vertices will all receive color $2$ since they form and independent set.

If $x\in V(P_6)$, where $P_6=w_1w_2w_3w_4w_5w_6$, then we distinguish three possibilities. If $x=w_1$ ($x=w_6$), then the graph $G-x$ is a disjoint union of $C_6$ and $P_5$. Alice's first move in an A-independence and AB-independence coloring game on $G-x$ is on vertex $w_4$ ($w_3$). Again, there are two more moves on $P_5$ with color $1$, which means that Alice can ensure that Bob will be the first to use color $1$ on $C_6$. Alice can then respond by coloring a vertex at distance two to Bob's choice, and color it also with color $1$. When they finish using color $1$, the remaining vertices will again all receive color $2$. Thus $\cigA(G-x)=\cigAB(G-x)=2$. The second possibility is that $x=w_2$ ($x=w_5$). In this case, the graph $G-x$ consists of $C_6$, $P_1$ and $P_4$. If Alice colors in her first move a vertex in $P_4$ or $P_1$, then two more moves are possible outside $C_6$ in the first round of the game, so Bob will again be the first to use color $1$ on $C_6$. For the same reason as in the previous case it follows that $\cigA(G-x)=\cigAB(G-x)=2$. The last possible option is that $x=w_3$ ($x=w_4$). Then the graph $G-x$ consists of $C_6$, $P_2$ and $P_3$. Since Alice wants Bob to be the first to make his move with color $1$ on $C_6$, she colors with color $1$ in her first move either the vertex $w_4$ ($w_1$) or the vertex $w_6$ ($w_3$). This will again force Bob to be the first to play on $C_6$ and hence $\cigA(G-x)=\cigAB(G-x)=2$.
\qed
\end{proof}

We further observe that there exist a lower, an upper and a mixed independence coloring game-vertex-critical graph. Moreover, the difference between the independence game chromatic number of a graph and the independence game chromatic number of its vertex deleted subgraph can be arbitrarily large.

\begin{proposition}
\label{independence-game-lower}
Let $n \geq 1$ be a positive integer. There exists a $\cigA$-lower-independence-coloring-game-vertex-critical and a $\cigAB$-lower-independence-coloring-game-vertex-critical graph $G$, such that $\cigA(G)-\cigA(G-x)=n$ and $\cigAB(G)-\cigAB(G-x)=n$ for every vertex $x \in V(G)$.
\end{proposition}

\begin{proof}
Let  $n \geq 1$ be a positive integer and $G=K_{n+2,n+2}-\mathcal{M}$, where $\mathcal{M}$ is a perfect matching. Denote the partite sets of $G$ with $A=\{a_1,a_2,\ldots ,a_{n+2}\}$ and $B=\{b_1,b_2,\ldots ,b_{n+2}\}$ and let $\mathcal{M}=\{a_1b_1,a_2b_2,\ldots ,a_{n+2}b_{n+2}\}$. We know that $\cigA(G)=\cigAB(G)=n+2$~\cite{db-2021}. Let $x\in V(G)$ be an arbitrary vertex. Assume that Alice and Bob play an A-independence or an AB-independence coloring game on $G-x$. Without loss of generality, suppose that $x=a_1$. Alice's first move is to color the vertex $b_1$ with color $1$. Then, no vertex of $A$ can be colored with color $1$ since $b_1$ is adjacent to all vertices of $A$ in $G-x$. Therefore, all vertices of $B$ will be colored in the first round of the game and all vertices of $A$ in the second round of the game. It follows that $\cigA(G-x)=\cigAB(G-x)=2$ and we get $\cigA(G)-\cigA(G-x)=\cigAB(G)-\cigAB(G-x)=(n+2)-2=n$. 
\qed
\end{proof}

\begin{proposition}
\label{independence-game-upper}
Let $n \geq 2$ be a positive integer. There exists a $\cigAB$-upper-game-vertex-critical graph $G$, such that
$$\left\{\cigAB(G)-\cigAB(G-x) \, | \, x \in V(G)\right\}=\{-1,-2n+3\}.$$
\end{proposition}

\begin{proof}
Let  $n \geq 2$ a positive integer and $G'$ be the disjoint union of graphs $G_1=(K_{2n,2n}-M)^u$ ($\mathcal{M}$ is a perfect matching in $K_{2n,2n}$) and $G_2=(K_{6,6}-\mathcal{M}')^v$ ($\mathcal{M}'$ is a perfect matching in $K_{6,6}$), where $u$ is the universal vertex in $G_1$ and $v$ is the universal vertex in $G_2$. Graph $G$ is obtained from $G'$ by identifying both universal vertices $u$ and $v$, i.e.\ $u=v$.

First we show that $\cigAB(G)=3$. The optimal first move for Alice in an AB-independence coloring game on $G$ using three colors is to color vertex $u$. Then Bob starts the second round of the game with color $2$ by playing a vertex in either $G_1-u$ or $G_2-u$, say $G_1-u$. Alice responds in her next move by coloring a vertex in the same partite set in $G_1-u$ as Bob played in his first move. Since there are even number of vertices in $G_1-u$, Bob is forced to make his first move with color $2$ in $G_2-u$. Alice again plays in her next move in the same partite set in $G_2-u$ as Bob did. After the second round of the game only one new color is needed to complete the game, since the remaining uncolored vertices form an independent set. Therefore, $\cigAB(G)=3$. 

If we remove the vertex $u$ from $G$ it does not matter where Alice plays in her first move. Without loss of generality assume that Alice makes her first move in $G_1-u$. Bob responds by coloring a vertex lying opposite Alice's colored vertex along the missing edge of $\mathcal{M}$. In this way, no other vertex in $G_1-u$ can receive color $1$. Hence, Alice's next move will be in $G_2-u$. Bob colors next a vertex lying opposite Alice's colored vertex along the missing edge of $\mathcal{M}'$ and the first round is complete. Alice starts the second round of the game and it goes exactly as the first round. It ends with two colored vertices in $G_1-u$ and two colored vertices in $G_2-u$. The continuation of the game goes along the same lines, because Alice starts each new round of the game. The game ends in $2n$ rounds and therefore $\cigAB(G-u)=2n$. This already shows that $\cigAB(G)-\cigAB(G-u)=-2n+3$.

We still need to show that the removal of all other vertices of $G$ also increases the independence game chromatic number of the vertex deleted subgraph. Let $x \in V(G)$ be an arbitrary vertex different from $u$. We will show that $\cigAB(G-x)=4$. If Alice played on $u$ in her first move of the game on $G-x$, Bob would start the second round of the game. He would play a vertex lying in the partite set of $G_1-u$ or $G_2-u$ to which vertex $x$ belongs. Thus, in this partite set, odd number of consecutive moves would be played, which means that Alice would be the first to color the vertices in the second one of the graphs $G_1-u$ and $G_2-u$. Bob would then play on the vertex opposite of Alice's last colored vertex along the missing edge of $\mathcal{M}$ or $\mathcal{M}'$. That would end the second round of the game. Clearly, in the continuation of the game at least four more rounds would be required. So playing $u$ is not an optimal first move for Alice.

It turns out that playing the vertex lying opposite of $x$ along the missing edge of $\mathcal{M}$ or $\mathcal{M}'$ is the optimal first move for Alice. Namely, in this case, in the graph of $G_1-u$ or $G_2-u$, in which $x$ lies, even number of moves are played in the first round (all vertices from the same partite set), so Alice has the first move in the second one of the graphs $G_1-u$ or $G_2-u$. Bob completes this round of the game by playing his next move on a vertex lying opposite to the vertex which Alice colored in her last move. Alice starts the next round by coloring a vertex lying in the same partite set that contained $x$. All vertices of this partite set must thus receive color $2$. Since there are an odd number of those vertices, Bob will be the first to play on the one graph $G_1-u$ or $G_2-u$, which does not contain $x$. Alice answers by coloring a vertex lying in the same partite set, thus causing all the vertices of that partite set to be colored in the second round. After the first two rounds there remains an uncolored set of pairwise non-adjacent vertices that can receive one color and the vertex $u$ that must receive its own color. So the game ends in four rounds. 

If Alice colors in her first move a vertex in the partite set that contained $x$ (without loss of generality say $x \in V(G_2-u)$), then the game also lasts four rounds. Bob clearly colors in his next move the vertex lying opposite Alice's colored vertex along the missing edge of $\mathcal{M}'$. Then Alice must color a vertex in $G_1-u$ and Bob plays again on the vertex lying opposite of Alice's last colored vertex along the missing edge of $\mathcal{M}$. This concludes the first round of the game. Alice starts the second round. She colors the vertex opposite of $x$ along the missing edge of $\mathcal{M}'$. Then all the remaining uncolored vertices of this partite set in $G_2-u$ must be colored in this round. Since there are an odd number of them, Bob will be the first to start coloring the vertices in $G_1-u$ with the color $2$. Alice then plays on a vertex in the same partite set of this graph in which Bob played last. This ensures that all vertices of this partite set will be colored in the second round. Henceforth, the game will last two more rounds. In one of the remaining rounds, vertex $u$ will be colored, and in the other, all the remaining uncolored vertices of $G$. 

Lastly, if Alice starts the game by playing on a vertex lying in the partite set opposite of $x$, but not exactly on the vertex lying opposite of $x$ along the missing edge of $\mathcal{M}'$, the strategy for Alice and Bob is exactly the same as in the previous case. The only difference is that Bob's and Alice's first moves swap. We see that $\cigAB(G-x)=4$, and consequently $\cigAB(G)-\cigAB(G-u)=3-4=-1$, which concludes the proof.
\qed
\end{proof}

\begin{proposition}
\label{independence-game-mixed}
Let $n \geq 2$ be a positive integer. There exists an $\cigAB$-mixed-game-vertex-critical graph $G$, such that
$$\left\{\cigAB(G)-\cigAB(G-x) \, | \, x \in V(G)\right\}=\{1,-2n+3,-2n+1\}.$$
\end{proposition}

\begin{proof}
Let $n\geq 2$ be a positive integer. Let $G_1=(K_{2n,2n}-\mathcal{M})^u$ ($\mathcal{M}$ is a perfect matching), where $u$ is the universal vertex in $G_1$, $G_2=(K_{2n+2,2n+2}-\mathcal{M}')^v$ ($\mathcal{M}'$ is a perfect matching), where $v$ is the universal vertex in $G_2$, and $G_3=(K_{2n+2,2n+2}-\mathcal{M}'')^w$ ($\mathcal{M}''$ is a perfect matching), where $w$ is the universal vertex in $G_3$. Graph $G$ is the graph obtained from graphs $G_1$, $G_2$, $G_3$ by adding edges $uv$, $vw$ and $uw$.

First we show that $\cigAB(G)=4$. Suppose that Alice and Bob play an AB-independence coloring game on $G$. Alice starts the first round of the game. Clearly, Alice should start the game by playing on one of the vertices $u, v$, or $w$. It does not matter on which of these vertices she starts the game. In all three cases the same number of colors are needed to complete the game. Therefore, suppose that Alice colors $u$ in her first move. Now, the remaining vertices in $G_1$, and vertices $v$ and $w$ can no longer receive color $1$. Thus, Bob colors in his first move one of the vertices in $G_2-v$ or $G_3-w$. Say he colors a vertex in $G_2-v$. Alice responds by coloring a vertex in the same partite set. After this move all vertices in this partite set must be colored in the first round of the game. Since the number of vertices in this partite set is even, Bob is forced to be the first to use color $1$ on a vertex in $G_3-w$. Again, Alice responds by coloring a vertex in the same partite set. With such a move she ensures that all vertices of this partite set are also colored in the first round of the game. When all the vertices in this partite set are colored, the first round of the game ends, as no vertex can be colored with color $1$. Since Alice makes the final move in the first round of the game, Bob begins the second round of the game. His goal is to force Alice to be the first to color an uncolored vertex of $G_1-u$, so in his first move he colors one of the vertices $v$ or $w$, say $v$ (note that, it does not matter on which of these two vertices he starts the second round of the game, as one of these two vertices will be colored with color $2$, the other with color $3$, and all the remaining uncolored vertices of graphs $G_2-v$ and $G_3-w$ will also be able to receive one of these two colors). After his first move in the second round no vertex of $G_2-v$ can be colored with color $2$. Because Alice does not want to be the first to color a vertex of $G_1-u$, she colors one of the vertices of $G_3-w$ which can still receive color $2$. Bob follows her in this move by coloring the vertex of the same partite set in his next move. Since they have even number of moves on this partite set and Alice started coloring it, she will be the first to color one of the uncolored vertices of $G_1-u$. Bob responds by coloring the opposite vertex along the missing edge of $\mathcal{M}$. After this move, the second round is over, since no other vertex can be colored with color $2$. Alice starts the third round of the game. She wants Bob to be the first to color the uncolored vertices in $G_1-u$, so she colors $w$ first. All uncolored vertices of $G_2-v$ must be colored in the third round of the game. Since Bob is the first who starts coloring them and there are an even number of these vertices, he will have the first move with color $3$ in the graph $G_3-w$. Alice responds by coloring a vertex in the same partite set. After this move, all uncolored vertices in this partite set can receive color $3$ and all the remaining uncolored vertices color $4$. We conclude that $\cigAB(G)=4$.

Next, we prove that $\cigAB(G-x)=3$ if $x\in V(G) \backslash\{u,v,w\}$. Assume that Alice and Bob play an AB-independence coloring game on $G-x$. The first move belongs to Alice. If $x \in G_1$, Alice's optimal first move is to color the vertex $u$, if $x\in G_2$ the vertex $v$ and if $x\in G_3$ the vertex $w$. Suppose that $x \in G_1$ and that Alice colors the vertex $u$ in her first move (the strategy for Alice to win the game on $G-x$ with three colors in the other two described cases is analogous). Then Bob has to make his first move on one of the partite set of the graphs $G_2-v$ or $G_3-w$, say $G_2-v$. Alice responds by coloring a vertex in the same partite set. Then all vertices in this partite set have to receive color $1$ and since their number is even, Bob will be the first to color the vertices in one of the partite sets of $G_3-w$ (note that vertices $v$ and $w$ can no longer receive color $1$, since their neighbor $u$ is colored with $1$). Again, Alice colors in her next move a vertex in the same partite set. Now, all vertices in this partite set must be colored in the first round of the game, and because their number is even, Alice makes the last move in this round. This means that the first move of the second round of the game belongs to Bob. If he colors some uncolored vertex in $G_1-u$, Alice responds by coloring an uncolored vertex in the same partite set in her next move. With such a move she ensures that all vertices in this partite set will receive color $2$. Otherwise, if Bob makes his first move with color $2$ anywhere else, Alice colors in her next move the vertex lying opposite of $x$ along the missing edge of $\mathcal{M}$. In this way, she ensures that all the vertices in the same partite set as the last colored vertex will also receive color $2$. If after Bob's response both $v$ and $w$ are still uncolored, Alice colors one of them (for at least one of them the color $2$ is allowed). After this move from Alice, it is clear which vertices will receive color $2$ in this round of the game, and all the remaining uncolored vertices will receive color $3$ in the next round. It follows that $\cigAB(G-x)=3$, and hence $\cigAB(G)-\cigAB(G-x)=1$.

It remains to prove that $\cigAB(G-x)=2n+1$ if $x=v$ or $x=w$ and that $\cigAB(G-x)=2n+3$ if $x=u$. If $x=v$ (or $x=w$), then Alice starts the first round of the game on $u$. With such a move she ensures that $2n+1$ colors are enough to finish the game on $G-x$. If Alice had started the game by coloring either $w$ or $v$, $2n+3$ colors would have be needed to complete the game. Otherwise, if $x=u$, then Alice stars the game by coloring either $v$ or $w$. With such a move she ensures that $2n+3$ colors are enough to finish the game on $G-x$. Since the continuation of the game is regardless of whether $x=u$, $x=v$ or $x=w$, we consider just the case when $x=v$. In this case, Alice colors the vertex $u$ in her first move of the game. Then, Bob has to color some vertex in one of the partite sets of the graphs $G_2-v$ or $G_3-w$, say $G_2-v$. Alice responds by coloring a vertex in the same partite set. After Alice's move, all vertices in this partite set have to receive color $1$. Since the number of vertices in this partite set is even, Bob is forced to take the first move in a partite set of $G_3-w$. Again, Alice responds by coloring the vertex in the same partite set. In the continuation of the first round of the game all vertices in this partite set receive color $1$. Since Alice made the final move in the first round, Bob begins the second round of the game. His optimal first move is to color vertex $w$. Then all uncolored vertices of the graph $G_2-v$ will have to be colored in the second round of the game. Since Alice has to be the first to color a vertex in $G_2-v$, and their number is even, Alice will be the first to color some uncolored vertex of the graph $G_1-u$ in this round. Bob responds by coloring the vertex lying opposite of Alice's last colored vertex along the missing edge of $\mathcal{M}$. After this move no other uncolored vertex can receive color $2$ and the second round of the game is over. In the third round of the game Alice has the first move. All uncolored vertices in $G_3-w$ can be colored in this round. Since their number is even, and Alice and Bob alternate turns, Bob colors the last of the uncolored vertices in $G_3-w$. Therefore, Alice is forced to start coloring the uncolored part of $G_1-u$. No matter where she plays, Bob respond by coloring the vertex lying opposite of Alice's last colored vertex along the missing edge of $\mathcal{M}$. After Bob's move, the third round of the game is over and $4n-4$ vertices of the graph $G_1-u$ remain uncolored. Since Alice has the first move in the new round of the game, only two vertices are colored in each of the subsequent rounds. Namely, Alice colors in each round of the game some uncolored vertex and Bob the vertex lying opposite of Alice's colored vertex along the missing edge of $\mathcal{M}$. Thus, another $2n-2$ rounds are required to complete the game. All together, the game lasts $2n-2+3=2n+1$ rounds. It follows that $\cigAB(G-v)=2n+1$, and $\cigAB(G)-\cigAB(G-v)=-2n+3$. In the case where $x=u$, we get $\cigAB(G-u)=2n+3$, and $\cigAB(G)-\cigAB(G-u)=-2n+1$.
\qed
\end{proof}

We conclude this section with a characterization of $2$-$\cigA$-game-vertex-critical ($2$-$\cigAB$-game-vertex-critical) and connected $3$-$\cigA$-lower-game-vertex-critical ($3$-$\cigAB$-lower-game-vertex-critical) graphs. In order to do this, we will need the characterization of connected graphs $G$ with $\cigA(G)=2$ ($\cigAB(G)=2$)~\cite{db-2021}, which will be heavily used in the proofs. Since this characterization is far from trivial, we state it as separate theorem.

\begin{theorem}[\!\!\cite{db-2021}]
\label{izrek_dva}
If $G$ is a connected graph with at least one edge, then the following statements are equivalent:
\begin{enumerate}[(1)]
\item $\cigA(G) =2$;
\item $\cigAB(G)=2$;
\item $G$ is a bipartite graph with the bipartition $V(G)=(X_1,X_2)$,
and there exists an $i\in\{1,2\}$ and a vertex $x$ in $X_i$, which is adjacent to all vertices from $X_j$, where $\{i,j\}=\{1,2\}$.
\end{enumerate}
\end{theorem}

Theorem~\ref{izrek_dva} introduces a vertex with a special property in bipartite graphs. Hence, if $G$ is a bipartite graph with the bipartition $V(G)=(X_1,X_2)$, and there exists an $i\in\{1,2\}$ and a vertex $x$ in $X_i$, which is adjacent to all vertices from $X_j$, $\{i,j\}=\{1,2\}$, then we will call $x$ a \emph{dominating vertex} in $G$.

\begin{proposition}
Let G be a connected graph. Graph $G$ is a $2$-$\cigA$-game-vertex-critical ($2$-$\cigAB$-game-vertex-critical) graph if and only if $G = K_2$.
\end{proposition}

\begin{proof}
Clearly, $\cigA(K_2)=\cigAB(K_2)=2$ and $\cigA(K_2-x)=\cigAB(K_2-x)=1$ for every vertex $x\in V(K_2)$.

Now assume that $G$ is a $2$-$\cigA$-vertex-critical ($2$-$\cigAB$-game-vertex-critical), connected graph. Since $\cigA(G)=2$ ($\cigAB(G)=2$), then $G$ is by Theorem~\ref{izrek_dva} a bipartite graph with the bipartition $V(G)=(X_1,X_2)$ for which there exists an $i \in \{1,2\}$ and a dominating vertex $x' \in X_i$. If $|X_j|=1$, $j \neq i$, and $y$ the unique vertex in $X_j$, then $G$ is a star graph with the central vertex $y$. If there exists a vertex $x \in V(X_i)$, $x \neq x'$, then $G-x$ is a also a star graph, and hence $\cigA(G-x)=\cigAB(G-x)=2$, which means that $G$ is not a $2$-$\cigA$-game-vertex-critical ($2$-$\cigAB$-game-vertex-critical) graph. The other case is that $x'$ is the only vertex in $X_i$, which gives $G=K_2$. However, if $|X_j| \geq 2$, then removing any vertex $y \in X_j$ yields a bipartite graph with the bipartition $V(G)=(X_1,X_2 \backslash \{y\})$ such that the vertex $x' \in X_i$ is adjacent to all vertices from $X_2 \backslash \{y\}$. Again by Theorem~\ref{izrek_dva}, $\cigA(G-y)=\cigAB(G-y)=2$ and $G$ is not a $2$-$\cigA$-game-vertex-critical ($2$-$\cigAB$-game-vertex-critical) graph.
\qed
\end{proof}

Before we characterize connected $3$-$\cigB$-lower-game-vertex-critical ($3$-$\cigAB$-lower-game-vertex-critical) graphs we need the following lemma. Even though it has a trivial proof, its use considerably simplifies both independence coloring games in the proof of Theorem~\ref{3-indep-lower}.

\begin{lemma}
\label{distance}
Assume that Alice and Bob play an A- or AB-independence coloring game on a connected graph $G$, and Alice colors in her first move a vertex $u \in V(G)$. If there exists a vertex $v \in V(G)$, such that $d(u,v) \geq 3$, then at least three different colors are needed to finish the game on $G$.
\end{lemma}

\begin{proof}
If Alice colors vertex $u \in V(G)$ with a color, then Bob chooses a vertex $w \in V(G)$ on a shortest path between $u$ and $v$ such that $d(u,w)=3$ (this is possible because $d(u,v) \geq 3$), and colors it with the same color. Both vertices that lie on a shortest path between $u$ and $w$ will have to receive each their own private color in the forthcoming rounds of the game, since they are adjacent, and one of them is also adjacent to $u$ and the other to $w$. Hence, at least three different colors are needed to finish the game on $G$.
\qed
\end{proof}

\begin{theorem}
\label{3-indep-lower}
Let G be a connected graph. Graph $G$ is a $3$-$\cigA$-lower-game-vertex-critical ($3$-$\cigAB$-lower-game-vertex-critical) graph if and only if one of the following holds:
\begin{enumerate}
\item $G=C_3$,
\item $G=C_5$,
\item $G=P_6$,
\item $G=K_{3,3}-\mathcal{M}$, where $\mathcal{M}$ is a perfect matching in $K_{3,3}$,
\item $G=C_4^+$, where $C_4^+$ is the graph obtained from $C_4$ by adding one pendant neighbor to each vertex of $C_4$.
\end{enumerate}
\end{theorem}

\begin{proof}
Obviously, $C_3$, $C_5$ and $P_6$ are $3$-$\cigA$-lower-game-vertex-critical and $3$-$\cigAB$-lower-game-vertex-critical graphs. 

Next, we show that $K_{3,3}-\mathcal{M}$ is $3$-$\cigA$-lower-game-vertex-critical and $3$-$\cigAB$-lower-game-vertex-critical graph. It is known that $\cigA(K_{3,3}-\mathcal{M})=\cigAB(K_{3,3}-\mathcal{M})=3$~\cite{db-2021}. When we delete an arbitrary vertex $x$ of $K_{3,3}-\mathcal{M}$ from one of the partite sets, there exists a vertex lying opposite of $x$ along the missing edge of $\mathcal{M}$, which is adjacent to both vertices of the partite set from which we removed $x$. This vertex is a dominating vertex and hence by Theorem~\ref{izrek_dva}, $\cigA((K_{3,3}-\mathcal{M})-x)=\cigAB((K_{3,3}-\mathcal{M})- x))=2$.

Clearly, $\cigA(C_4^+)=\cigAB(C_4^+)=3$ since wherever Alice starts the first round of the game, there always exists a vertex at the distance $3$ of Alice's chosen vertex, hence by Lemma~\ref{distance}, at least three different colors are needed to finish the game on $C_4^+$. When we remove an arbitrary vertex $x \in C_4^+$, there exists a vertex $y$ in graph $C_4^+-x$ which has distance at most $2$ to every other vertex of the connected component of $C_4^+-x$ to which $y$ belongs. Note that the other connected component of $C_4^+-x$, if it exists, consists of a single vertex. If Alice colors $y$ in the first round of the game played on $C_4^+-x$, the game ends in two rounds. Therefore, $\cigA(C_4^+-x)=\cigAB(C_4^+-x)=2$ for an arbitrary $x\in V(C_4^+)$. 

To prove the other direction assume that $G$ is a $3$-$\cigA$-lower-game-vertex-critical ($3$-$\cigAB$-lower-game-vertex-critical) graph. Note that if $G$ contains an odd cycle as a proper subgraph, then we can remove a vertex from $G$ that does not belong to this cycle, and the remaining graph will still contain an odd cycle as a subgraph. Hence, at least three colors will be needed to color it.

We remove an arbitrary vertex $x \in V(G)$. We may assume that $\cigA(G)=\cigAB(G)=3$ and $\cigA(G-x)=\cigAB(G-x)=2$. Namely, if $\cigA(G-x)=\cigAB(G-x)=1$, then $G-x$ is formed from isolated vertices, and since $G$ is connected it must be a star graph. Thus $\cigA(G)=\cigAB(G)=2$, which is a contradiction. Since $G-x$ requires two colors to be properly colored, it must be a disjoint union of connected bipartite graphs $B_1, \ldots, B_k$, $k \geq 1$, and isolated vertices $z_1, \ldots ,z_\ell$, $\ell \geq 0$. The vertex $x$ must be adjacent to all bipartite graphs and all isolated vertices in $G$, since $G$ is by assumption a connected graph. For every $i \in \{1, \ldots , k\}$ let us denote both partite sets of $B_i$ with $X_i$ and $Y_i$. It is clear that $|X_i| \geq 1$ and $|Y_i| \geq 1$ for every $i \in \{1, \ldots , k\}$, since for otherwise the vertices of those partite sets would belong to isolated vertices. We split the proof into three cases.

\medskip

\textbf{Case 1:} $k \geq 3$.\\
If there exists an $i \in \{1, \ldots , k\}$ such that $x$ is adjacent to vertices of both partite sets $X_i$ and $Y_i$ in $G$, then $G$ must contain an odd cycle as a proper subgraph in which case it can not be a $3$-$\cigA$-lower-game-vertex-critical ($3$-$\cigAB$-lower-game-vertex-critical) graph. Thus, $x$ can be adjacent to the vertices of only one partite set of every bipartite subgraph of $G$. Without loss of generality we may assume that $x$ is adjacent only to the vertices of $X_i$ for every $i \in \{1, \ldots , k\}$. If $x$ is adjacent to all vertices of $X_i$ for every $i \in \{1, \ldots , k\}$, then it is a dominating vertex in $G$ and by Theorem~\ref{izrek_dva} we have $\cigA(G)=\cigAB(G)=2$, which is a contradiction. Hence, there exists an $i \in \{1, \ldots , k\}$, and a vertex $x' \in X_i$, such that $xx' \notin E(G)$. Now we remove a vertex $v \in V(B_j)$, $j \neq i$, such that the remaining graph $G-v$ stays connected. Since $G-v$ is still a bipartite graph with no dominating vertex, it has $\cigA(G-v) \neq 2$ and $\cigAB(G-v) \neq 2$ by Theorem~\ref{izrek_dva}. Since clearly also $\cigA(G-v) \neq 1$ and $\cigAB(G-v) \neq 1$, we have $\cigA(G-v) \geq 3$ and $\cigAB(G-v) \geq 3$, which is a contradiction.

\medskip

\textbf{Case 2:} $k=2$.\\
If there exists an $i \in \{1,2\}$ such that $x$ is adjacent to vertices of both partite sets $X_i$ and $Y_i$ in $G$, then $G$ must contain an odd cycle as a proper subgraph in which case it can not be a $3$-$\cigA$-lower-game-vertex-critical ($3$-$\cigAB$-lower-game-vertex-critical) graph. Thus, $x$ can be adjacent to the vertices of only one partite set, say $X_1$ and $X_2$. If $x$ is adjacent to all vertices of $X_1$ and $X_2$, then it is a dominating vertex and by Theorem~\ref{izrek_dva} we have $\cigA(G)=\cigAB(G)=2$, which is a contradiction. Therefore, there exists an $i \in \{1,2\}$, and a vertex $x' \in X_i$, such that $xx' \notin E(G)$.

\medskip

\textbf{Subcase 2.1:} $\ell \geq 1$.\\
We remove the vertex $z_1$ to obtain the graph $G-z_1$. Since $G-z_1$ is still a bipartite graph that does not have a dominating vertex, if follows from Theorem~\ref{izrek_dva} that $\cigA(G-z_1) \neq 2$ and $\cigAB(G-z_1) \neq 2$. Since clearly also $\cigA(G-z_1) \neq 1$ and $\cigAB(G-z_1) \neq 1$, we have $\cigA(G-z_1) \geq 3$ and $\cigAB(G-z_1) \geq 3$, which is again a contradiction.

\medskip

\textbf{Subcase 2.2:} $\ell = 0$.\\
If $|Y_j| \geq 2$, $j \neq i$, than removing any vertex $v \in V(B_j)$, such that the remaining graph $G-v$ stays connected, yields a bipartite graph with no dominating vertex, which by Theorem~\ref{izrek_dva} gives $\cigA(G-v) \neq 2$ and $\cigAB(G-v) \neq 2$. Since also $\cigA(G-v) \neq 1$ and $\cigAB(G-v) \neq 1$, we have $\cigA(G-v) \geq 3$ and $\cigAB(G-v) \geq 3$, which is a contradiction. Hence, $|Y_j|=1$. Let us denoted with $y_j$ the only vertex in $Y_j$. If also $|X_j| \geq 2$, then removing any vertex $u \in X_j$, such that the remaining graph $G-u$ stays connected, again yields a bipartite graph with no dominating vertex, which by Theorem~\ref{izrek_dva} gives $\cigA(G-u) \neq 2$ and $\cigAB(G-u) \neq 2$. Since also $\cigA(G-u) \neq 1$ and $\cigAB(G-u) \neq 1$, we have $\cigA(G-u) \geq 3$ and $\cigAB(G-u) \geq 3$, which is a contradiction. Thus, $|X_j|=1$. Let us denoted with $x_j$ the only vertex in $X_j$. Clearly, $xx_j \in E(G)$ and $x_jy_j \in E(G)$, since $G$ is connected. Similarly, if $|Y_i| \geq 2$, then we remove a vertex $v \in V(B_i) \backslash \{x'\}$, such that the remaining graph $G-v$ stays connected. Note that this is possible because of $|Y_i| \geq 2$. Namely, if $|Y_i|=1$, then $x'$ might be the only vertex that can be removed, and this is precisely what we do not want. The graph $G-v$ is bipartite and does have a dominating vertex, which by Theorem~\ref{izrek_dva} gives $\cigA(G-v) \neq 2$ and $\cigAB(G-v) \neq 2$. Since also $\cigA(G-v) \neq 1$ and $\cigAB(G-v) \neq 1$, we have $\cigA(G-v) \geq 3$ and $\cigAB(G-v) \geq 3$, which is a contradiction. Hence, $|Y_i|=1$. Let us denoted with $y_i$ the only vertex in $Y_i$. Also, if $|X_i| \geq 3$, then removing any vertex $u \in X_j \backslash \{x'\}$, such that the remaining graph $G-u$ stays connected, again yields a bipartite graph with no dominating vertex, which by Theorem~\ref{izrek_dva} gives $\cigA(G-u) \neq 2$ and $\cigAB(G-u) \neq 2$. Since also $\cigA(G-u) \neq 1$ and $\cigAB(G-u) \neq 1$, we have $\cigA(G-u) \geq 3$ and $\cigAB(G-u) \geq 3$, which is a contradiction. Thus, $|X_i|=2$. Let us denoted with $x_i$ the only vertex in $X_i \backslash \{x'\}$. Clearly, $xx_i \in E(G)$, $x_iy_i \in E(G)$ and $x'y_i \in E(G)$ since $G$ is connected. What we obtained is $G=P_6$, which is one of the $3$-$\cigA$-lower-game-vertex-critical ($3$-$\cigAB$-lower-game-vertex-critical) graphs.

\medskip

\textbf{Case 3:} $k=1$.\\
From this point on let $X=X_1$ and $Y=Y_1$ denote both partite sets of $B_1$. If $ \ell \geq 1$ and $x$ is adjacent to vertices of both partite sets $X$ and $Y$ in $G$, then $G$ must contain an odd cycle as a proper subgraph in which case it can not be a $3$-$\cigA$-lower-game-vertex-critical ($3$-$\cigAB$-lower-game-vertex-critical) graph. Thus, $x$ can be adjacent to the vertices of only one partite set, say $X$. If $x$ is adjacent to all vertices of $X$, then it is dominating vertex and by Theorem~\ref{izrek_dva} we have $\cigA(G)=\cigAB(G)=2$, which is a contradiction. Therefore, there exists a vertex $x' \in X$, such that $xx' \notin E(G)$.

\medskip

\textbf{Subcase 3.1:} $\ell \geq 2$.\\
We remove the vertex $z_1$ to obtain the graph $G-z_1$. Since $G-z_1$ is still a bipartite graph with no dominating vertex, Theorem~\ref{izrek_dva} gives $\cigA(G-z_1) \neq 2$ and $\cigAB(G-z_1) \neq 2$. Since clearly also $\cigA(G-z_1) \neq 1$ and $\cigAB(G-z_1) \neq 1$, we have $\cigA(G-z_1) \geq 3$ and $\cigAB(G-z_1) \geq 3$, which is again a contradiction.

\medskip

\textbf{Subcase 3.2:} $\ell = 1$.\\
Let us denote with $z=z_1$ the only isolated vertex in $G-x$, and let $x_1, \ldots ,x_a$, $a \geq 1$, be the vertices in $X \backslash \{x'\}$ that are adjacent to $x$ (note that $a=0$ is not possible since $G$ is connected). Since $G$ is bipartite and $\cigA(G)=\cigAB(G)=3$, $x_i$'s can not be dominating vertices. However, since $G-z$ is also a connected bipartite graph, and we want $\cigA(G-z)=\cigAB(G-z)=2$, $G-z$ must have a dominating vertex. By construction, vertices of $X$ can not be dominating vertices of $G-z$, hence there exists a vertex $y' \in Y$ such that $y'$ is adjacent to every vertex of $X$. Since $x_i$ is not a dominating vertex in $G$, there exists $y_i \in Y$ such that $x_iy_i \notin E(G)$ for every $i \in \{1, \ldots ,a\}$. Note that some of the vertices $y_i$ might represent the same vertex. Since vertices $x_i$ and $y_i$ are not adjacent in the bipartite graph $G$, we have $d_G(x_i,y_i) \geq 3$. Moreover, since $y' \in Y$ is a dominating vertex, it follows that $d_G(x_i,y_i)=3$ for every $i \in \{1, \ldots ,a\}$. For every $i \in \{1, \ldots ,a\}$ we find a shortest path $L_i$, $|V(L_i)|=3$, in $G$ between vertices $x_i$ and $y_i$ ($x_i$ and $y_i$ are the endvertices of $L_i$). Note that some of the paths $L_i$ might intersect each other. We define the set
$$S=V(G) \backslash \left(\bigcup_{i=1}^{a}V(L_i) \cup \{z,x,y',x'\}\right).$$
If $S \neq \emptyset$, then we remove a vertex $u \in S$ such that the graph $G-u$ remains connected. Suppose that Alice and Bob play an A- or AB-independence coloring game on the graph $G-u$. If Alice colors a vertex $v \in V(G-u) \backslash \{z,x,x_1, \ldots ,x_a\}$ in her first move, then $d_{G-u}(v,z) \geq 3$; if Alice colors $x_i$, $i \in \{1, \ldots ,a\}$, in her first move, then $d_{G-u}(x_i,y_i) \geq 3$; if Alice colors $x$ in her first move, then $d_{G-u}(x,x')=3$; if Alice colors $z$ in her first move, then $d_{G-u}(z,y')=3$. Considering all possibilities, and using Lemma~\ref{distance}, at least three different colors are needed to finish the game on $G-u$, which is a contradiction. We may therefore assume that $S = \emptyset$, which means that
$$V(G)=\left(\bigcup_{i=1}^{a}V(L_i) \cup \{z,x,y',x'\}\right).$$
We define the set $T=V(G) \backslash \{z,x,x_1, \ldots, x_a, y_1, \ldots, y_a,y',x'\}$, and split the proof with respect to the positive integer $a \geq 1$.

Suppose first that $a=1$. If $x' \notin V(L_1)$, then $G-x'$ is connected, and we let Alice and Bob play an A- or AB-independence coloring game on the graph $G-x'$. If Alice colors a vertex $v \in V(G-x') \backslash \{z,x,x_1\}$ in her first move, then $d_{G-x'}(v,z) \geq 3$; if Alice colors either $x_1$ or $x$ in her first move, then $d_{G-x'}(x_1,y_1)=3$ and $d_{G-x'}(x,y_1)=4$ (note that the shortest path $L_1$ between $x$ and $y_1$ goes through $x_1$); if Alice colors $z$ in her first move, then $d_{G-x'}(z,y')=3$. Considering all possibilities, and using Lemma~\ref{distance}, at least three different colors are needed to finish the game on $G-x'$, which is a contradiction. However, if $x' \in V(L_1)$, then by the structure of $G$, and $d_G(x_1,y_1)=3$, $x'$ must be adjacent to $y_1$. If $T \neq \emptyset$, then there exists exactly one vertex $u \in T$. By the structure of $G$, $G-u$ is the path $P_6$, and since $\cigA(P_6)=\cigAB(P_6)=3$, $G$ can not be a $3$-$\cigA$-lower-game-vertex-critical ($3$-$\cigAB$-lower-game-vertex-critical) graph. Therefore, $T = \emptyset$. In this case $V(G)=\{z,x,x_1,y_1,y',x'\}$, and obeying all the given adjacencies, we get $G=P_6$, which is one of the $3$-$\cigA$-lower-game-vertex-critical ($3$-$\cigAB$-lower-game-vertex-critical) graphs.

Now suppose that $a \geq 2$. First suppose that there exists $j \in \{1, \ldots ,a\}$ such that $G-x_j$ is connected (note that in the case $a=1$, $G-x_1$ is always disconnected). We let Alice and Bob play an A- or AB-independence coloring game on the graph $G-x_j$. If Alice colors a vertex $v \in V(G-x_j) \backslash \{z,x,x_1, \ldots , x_{j-1},x_{j+1}, \ldots ,x_a\}$ in her first move, then $d_{G-x_j}(v,z) \geq 3$; if Alice colors $x_i$, $i \in \{1, \ldots ,a\} \backslash \{j\}$, in her first move, then $d_{G-x_j}(x_i,y_i)=3$; if Alice colors $x$ in her first move, then $d_{G-x_j}(x,x')=3$; if Alice colors $z$ in her first move, then $d_{G-x_j}(z,y')=3$. Considering all possibilities, and using Lemma~\ref{distance}, at least three different colors are needed to finish the game on $G-x_j$, which is a contradiction. Therefore, we may assume that $G-x_i$ is disconnected for every $i \in \{1, \ldots ,a\}$. By the structure of $G$, $G-y_i$ is connected for every $i \in \{1, \ldots ,a\}$ ($y_i$ is an endvertex of the path $L_i$). Now suppose that for some $j \in \{1, \ldots ,a\}$ there exists an $h \in \{1, \ldots ,a\}$, $h \neq j$, such that $x_jy_h \notin E(G)$. Then $d_G(x_j,y_h) \geq 3$. Suppose that Alice and Bob play an A- or AB-independence coloring game on the graph $G-y_j$. If Alice colors a vertex $v \in V(G-y_j) \backslash \{z,x,x_1, \ldots ,x_a\}$ in her first move, then $d_{G-y_j}(v,z) \geq 3$; if Alice colors $x_i$, $i \in \{1, \ldots ,a\} \backslash \{j\}$, in her first move, then $d_{G-y_j}(x_i,y_i)=3$; if Alice colors $x_j$ in her first move, then $d_{G-y_j}(x_j,y_h) \geq 3$; if Alice colors $x$ in her first move, then $d_{G-y_j}(x,x')=3$; if Alice colors $z$ in her first move, then $d_{G-y_j}(z,y')=3$. Considering all possibilities, and using Lemma~\ref{distance}, at least three different colors are needed to finish the game on $G-y_j$, which is a contradiction. Thus we may assume that for every $i \in \{1, \ldots ,a\}$ and every $j \in \{1, \ldots ,a\}$, $j \neq i$, $x_iy_j \in E(G)$. This means that if $a \geq 3$, then $G-x_i$ is connected for every $i \in \{1, \ldots ,a\}$, which is a contradiction, since we assumed that $G-x_i$ is disconnected for every $i \in \{1, \ldots ,a\}$. Hence, $a=2$. If $T \neq \emptyset$, then we remove a vertex $u \in T$ such that the graph $G-u$ remains connected. Suppose that Alice and Bob play an A- or AB-independence coloring game on the graph $G-u$. If Alice colors a vertex $v \in V(G-u) \backslash \{z,x,x_1,x_2\}$ in her first move, then $d_{G-u}(v,z) \geq 3$; if Alice colors $x_1$ in her first move, then $d_{G-u}(x_1,y_1)=3$; if Alice colors $x_2$ in her first move, then $d_{G-u}(x_2,y_2)=3$; if Alice colors $x$ in her first move, then $d_{G-u}(x,x')=3$; if Alice colors $z$ in her first move, then $d_{G-u}(z,y')=3$. Considering all possibilities, and using Lemma~\ref{distance}, at least three different colors are needed to finish the game on $G-u$, which is a contradiction. Therefore, $T = \emptyset$. In this case $V(G)=\{z,x,x_1,x_2,y_1,y_2,y',x'\}$, and obeying all the given adjacencies, we get $G=C_4^+$, which is one of the $3$-$\cigA$-lower-game-vertex-critical ($3$-$\cigAB$-lower-game-vertex-critical) graphs.

\medskip

\textbf{Subcase 3.3:} $\ell = 0$.\\
We see that $G-x$ must be a connected and bipartite graph. We again denote with $X=X_1$ and $Y=Y_1$ both partite sets of $G-x$. If $x$ is adjacent to vertices of both partite sets $X$ and $Y$ in $G$, then $G$ must contain an odd cycle. Let $C_{2n+1}$, $n \geq 1$, be a smallest odd (induced) cycle in $G$. If $G$ contains $C_{2n+1}$ as a proper subgraph, then it is not a $3$-$\cigA$-lower-game-vertex-critical ($3$-$\cigAB$-lower-game-vertex-critical) graph. If we remove a vertex in $G$ that does not belong to $C_{2n+1}$, then at least three colors will be required to complete the game. Therefore, we may assume that $G=C_{2n+1}$. If $n=1$ or $n=2$, then $G=C_3$ and $G=C_5$ are $3$-$\cigA$-lower-game-vertex-critical ($3$-$\cigAB$-lower-game-vertex-critical) graphs. Let us consider the remaining case $n \geq 3$, and denote $G=C_{2n+1}=v_1v_2\ldots v_{2n+1}$, where $v_iv_{i+1} \in E(G)$ for all $i \in \{1, \ldots , 2n+1\}$ (modulo $2n+1$). Assume that Alice and Bob play an A- or AB-independence coloring game on $G-v_{2n+1}$, and suppose that Alice colors vertex $v_i$. Because of symmetry we may assume that $i \leq n$. Clearly, $v_i$ ($v_{i+1}$) can not be adjacent to $v_{i+2}$ ($v_{i+3}$), since $G$ would contain $C_3$ as a subgraph, which is not possible. Also, $v_i$ can not be adjacent to $v_{i+3}$, because $v_1 \ldots v_iv_{i+3} \ldots v_{2n+1}$ would form a $(2n-1)$-cycle in $G$, which is also not possible, since $C_{2n+1}$ is a smallest cycle in $G$. This means that $d_G(v_i,v_{i+3})=3$, and by Lemma~\ref{distance}, at least three different colors are needed to finish the game on $G$. Therefore, $G=C_{2n+1}$, $n \geq 3$, is not a $3$-$\cigA$-lower-game-vertex-critical ($3$-$\cigAB$-lower-game-vertex-critical) graph. 

What remains to consider is the case, when $x$ is adjacent to vertices of only one partite sets $X$ and $Y$, say $X$. In this case, $G$ is clearly also a bipartite graph, and we may consider the vertex $x$ as one of the vertices which belong to the partite set $Y$. Let us denote with $x_1, \ldots ,x_a$, $a \geq 1$, the vertices in $X \backslash \{x'\}$ that are adjacent to $x$ (note that again $a=0$ is not possible since $G$ is connected). Since $G$ is bipartite and $\cigA(G)=\cigAB(G)=3$, $x_i$'s can not be dominating vertices. However, since $G-x$ is also a bipartite graph, and we want $\cigA(G-x)=\cigAB(G-x)=2$, $G-x$ must by Theorem~\ref{izrek_dva} have a dominating vertex. Since we can consider the vertex $x$ as one of the vertices which belong to the partite set $Y$, vertices of $Y$ can not be dominating vertices of $G-x$, for otherwise a dominating vertex of $G-x$ would also be a dominating vertex of $G$, which is a contradiction to the fact that $\cigA(G)=\cigAB(G)=3$. Therefore, there must exist a vertex in $X$ which is not adjacent to $x$ and is adjacent to every vertex of $Y$. Without loss of generality let that be our predefined vertex $x' \in X$. Since $x_i$'s are not dominating vertices in $G$, there exists $y_i \in Y$ such that $x_iy_i \notin E(G)$ for every $i \in \{1, \ldots ,a\}$, and we have $d_G(x_i,y_i) \geq 3$ for every $i \in \{1, \ldots ,a\}$. Note that some of the vertices $y_i$'s might represent the same vertex. Moreover, since $x' \in X$ is a dominating vertex, it follows that $d_G(x_i,y_i)=3$ for every $i \in \{1, \ldots ,a\}$ and hence, there exists a path $L_i$, $|V(L_i)|=3$, between vertices $x_i$ and $y_i$ for every $i \in \{1, \ldots ,a\}$. Similarly, we observe that $d_G(x,x')=3$. Namely, since $xx'\notin E(G)$ and $G$ is bipartite, we have $d_G(x,x')\geq 3$. Since $x'$ is a dominating vertex in $G-x$ and $a\neq 0$, there exists a vertex $y'\in V(G-x)$ such that $x'y'\in E(G-x)$ and $y'x_1\in E(G-x)$. We can also observe that if there exists a vertex $v \in Y \backslash \{y_1, \ldots ,y_a\}$ such that $G-v$ is connected, then $\cigA(G-v)=\cigAB(G-v) \geq 3$. Namely, if Alice colors a vertex $w \in V(G-v) \backslash \{x,x',x_1, \ldots ,x_a, y_1, \ldots y_a\}$, $w\in Y$, in her first move, then there exists a vertex $x''\in X$ such that $d_{G-v}(w,x'') \geq 3$ for otherwise $w$ would be a dominating vertex in $G$, which contradicts the assumption that the vertices of $Y$ cannot be dominating vertices in $G$; if Alice colors a vertex $w \in V(G-v) \backslash \{x,x',x_1, \ldots ,x_a, y_1, \ldots y_a\}$, $w\in X$, in her first move, then $d_{G-v}(w,x) \geq 3$; if Alice colors $x$ or $x'$ in her first move, then $d_{G-v}(x,x')=3$; if Alice colors $x_i$ or $y_i$, $i \in \{1, \ldots ,a\}$, in her first move, then $d_{G-v}(x_i,y_i)=3$. Considering all possibilities, and using Lemma~\ref{distance}, at least three different colors are needed to finish the game on $G-v$, which is a contradiction. From now on we may assume that $G-v$ is disconnected for every $v \in Y \backslash \{y_1, \ldots ,y_a\}$. We split the proof with respect to the positive integer $a \geq 1$.

Suppose first that $a=1$. Observe that there exists a vertex $u \in X$ such that $d_G(y',u)=3$ (otherwise $y'$ would be a dominating vertex, which is a contradiction since there are no dominating vertices in $Y$).

First consider the case when $y_1u\in E(G)$. If there exists a vertex $y'' \in Y \backslash \{y_1,y'\}$, then we assumed that $G-y''$ is disconnected. This means that $y''$ has a pendant neighbor $x'' \in X$. Namely, if there existed a vertex $v \in Y$, $v \neq y''$, such that $x''v \in E(G)$, then $G-y''$ would be connected because $x'$ is adjacent to every vertex of $Y$. Since $x''$ is a pendant vertex in $G$, $G-x''$ is connected. Similarly, if there is a vertex of $Y \backslash \{y_1,y'\}$ that is adjacent to $x_1$, then $y'$ must also have a pendant neighbor in $X$, since we assumed that $G-y'$ is disconnected. Suppose that Alice and Bob play an A- or AB-independence coloring game on the graph $G-x''$. If Alice colors a vertex $w \in V(G-x'') \backslash \{x,x_1,y',x',y_1,u\}$, $w \in Y$, in her first move, then either $d_{G-x''}(w,x_1)=3$ (if $w$ and $x_1$ are not adjacent) or $w$ has distance $3$ to a pendant neighbor of $y'$ (such a neighbor exists if $w$ and $x_1$ are adjacent); if Alice colors a vertex $w \in V(G-x'') \backslash \{x,x_1,y',x',y_1,u\}$, $w\in X$, in her first move, then $d_{G-x''}(w,x) \geq 3$; if Alice colors $x$ or $x'$ in her first move, then $d_{G-x''}(x,x')=3$; if Alice colors $x_1$ or $y_1$ in her first move, then $d_{G-x''}(x_1,y_1)=3$; if Alice colors $y'$ or $u$ in her first move, then $d_{G-x''}(y',u)=3$. Considering all possibilities, and using Lemma~\ref{distance}, at least three different colors are needed to finish the game on $G-x''$, which is a contradiction. This means that we can consider only the case when $Y = \{y_1,y'\}$.  Hence, if $V(G) \backslash \{u,x,x_1,y_1,y',x'\} \neq \emptyset$, then there exists a vertex $v \in V(G) \backslash \{u,x,x_1,y_1,y',x'\}$ that also belongs to $X$, and since $Y = \{y_1,y'\}$, $G-v$ must be connected. Suppose that Alice and Bob play an A- or AB-independence coloring game on the graph $G-v$. If Alice colors a vertex $w \in V(G-v) \backslash \{x,x_1,y',x',y_1,u\}$, $w\in X$, in her first move, then $d_{G-v}(w,x) \geq 3$; if Alice colors $x$ or $x'$ in her first move, then $d_{G-v}(x,x')=3$; if Alice colors $x_1$ or $y_1$ in her first move, then $d_{G-v}(x_1,y_1)=3$; if Alice colors $y'$ or $u$ in her first move, then $d_{G-v}(y',u)=3$. Considering all possibilities, and using Lemma~\ref{distance}, at least three different colors are needed to finish the game on $G-v$, which is a contradiction. Therefore, $V(G)=\{x,x_1,y_1,y',x',u\}$, and obeying all the given adjacencies, we get $G=P_6$, which is one of the $3$-$\cigA$-lower-game-vertex-critical ($3$-$\cigAB$-lower-game-vertex-critical) graphs.

Next consider the case when $y_1u \notin E(G)$. Then there exists a vertex $y''\in Y$ such that $y''u\in E(G)$. If there also exists a vertex $y''' \in Y \backslash \{y_1,y',y''\}$, then we assumed that $G-y'''$ is disconnected. Analogues to the previous case, $y'''$ has a pendant neighbor $x'' \in X$. Similarly, if there is a vertex of $Y \backslash \{y_1,y',y''\}$ that is adjacent to $x_1$, then $y'$ must also have a pendant neighbor in $X$, since we assumed that $G-y'$ is disconnected. Suppose that Alice and Bob play an A- or AB-independence coloring game on the graph $G-x''$. If Alice colors a vertex $w \in V(G-x'') \backslash \{x,x_1,y',x',y_1,u\}$, $w \in Y$, in her first move, then either $d_{G-x''}(w,x_1)=3$ (if $w$ and $x_1$ are not adjacent) or $w$ has distance $3$ to a pendant neighbor of $y'$ (such a neighbor exists if $w$ and $x_1$ are adjacent); if Alice colors a vertex $w \in V(G-x'') \backslash \{x,x_1,y',x',y_1,u\}$, $w\in X$, in her first move, then $d_{G-x''}(w,x) \geq 3$; if Alice colors $x$ or $x'$ in her first move, then $d_{G-x''}(x,x')=3$; if Alice colors $x_1$ or $y_1$ in her first move, then $d_{G-x''}(x_1,y_1)=3$; if Alice colors $y'$ or $u$ in her first move, then $d_{G-x''}(y',u)=3$. Considering all possibilities, and using Lemma~\ref{distance}, at least three different colors are needed to finish the game on $G-x''$, which is a contradiction. This means that we can consider only the case when $Y = \{y_1,y',y''\}$. If $x_1y''\notin E(G)$ and $V(G) \backslash \{u,x,x_1,y_1,y',x',y''\} \neq \emptyset$, then there exists a vertex $v \in V(G) \backslash \{u,x,x_1,y_1,y',x',y''\}$ that also belongs to $X$ and since $Y = \{y_1,y',y''\}$, $G-v$ must be connected. Suppose that Alice and Bob play an A- or AB-independence coloring game on the graph $G-v$. If Alice colors a vertex $w \in V(G-v) \backslash \{x,x_1,y',x',y_1,u, y''\}$, $w\in X$, in her first move, then $d_{G-v}(w,x) \geq 3$; if Alice colors $x$ or $x'$ in her first move, then $d_{G-v}(x,x')=3$; if Alice colors $x_1$ or $y_1$ in her first move, then $d_{G-v}(x_1,y_1)=3$; if Alice colors $y'$ or $u$ in her first move, then $d_{G-v}(y',u)=3$; if Alice colors $y''$ in her first move, then $d_{G-v}(y'',x_1)=3$. Considering all possibilities, and using Lemma~\ref{distance}, at least three different colors are needed to finish the game on $G-v$, which is a contradiction. If $x_1y''\notin E(G)$ and $V(G)=\{u,x,x_1,y_1,y',x',y''\}$, then $G-y_1$ is the path $P_6$ and since $\cigA(P_6)=\cigAB(P_6)=3$, $G$ can not be a $3$-$\cigA$-lower-game-vertex-critical ($3$-$\cigAB$-lower-game-vertex-critical) graph. What remains to consider is the case $x_1y''\in E(G)$. Since we assumed that $G-y'$ must be disconnected, then $y'$ must have a pendant neighbor $x'' \in X$. If $V(G) \neq \{u,x,x_1,y_1,y',x',y'',x''\}$, then there exists a vertex $v \in V(G) \backslash \{u,x,x_1,y_1,y',x',y'',x''\}$ that also belongs to $X$, which means that $G-v$ is again connected. Suppose that Alice and Bob play an A- or AB-independence coloring game on the graph $G-v$. If Alice colors a vertex $w \in V(G-v) \backslash \{x,x_1,y',x',y_1,u, y'',x''\}$, $w\in X$, in her first move, then $d_{G-v}(w,x) \geq 3$; if Alice colors $x$ or $x'$ in her first move, then $d_{G-v}(x,x')=3$; if Alice colors $x_1$ or $y_1$ in her first move, then $d_{G-v}(x_1,y_1)=3$; if Alice colors $y'$ or $u$ in her first move, then $d_{G-v}(y',u)=3$; if Alice colors $y''$ or $x''$ in her first move, then $d_{G-v}(y'',x'')=3$. Considering all possibilities, and using Lemma~\ref{distance}, at least three different colors are needed to finish the game on $G-v$, which is a contradiction. Hence $V(G)=\{u,x,x_1,y_1,y',x',y'',x''\}$, and obeying all the given adjacencies, we get $G=C_4^+$, which is one of the $3$-$\cigA$-lower-game-vertex-critical ($3$-$\cigAB$-lower-game-vertex-critical) graphs.

Finally, let $a \geq 2$. The graph $G-x_i$ is connected for every $i \in \{1, \ldots ,a\}$, because every neighbor of $x_i$, except $x$, is adjacent to $x'$, and since $a \geq 2$, vertices $x$ and $x'$ are connected by at least two paths of order $3$, one going through $x_i$, $i \in \{1, \ldots ,a\}$, and the other through some $x_j$, $j \neq i$. If there exists a vertex $v \in Y \backslash \{y_1, \ldots ,y_a\}$, then we assumed that $G-v$ is disconnected. This means that $v$ has a pendant neighbor $u \in X$. Since the vertex $v \in Y$ can not be a dominating vertex in $G$, there exists a vertex $x'' \in X$ such that $vx'' \notin E(G)$. Let Alice an Bob play an A- or AB-independence coloring game on the graph $G-x_1$. Without loss of generality we may assume that $x'' \neq x_1$. If $x_1=x''$, then we let Alice an Bob play an A- or AB-independence coloring game on the graph $G-x_2$, and the proof goes along the same lines. If Alice colors a vertex $w \in V(G-x_1) \backslash \{v,x,x',x_2, \ldots ,x_a, y_1, \ldots ,y_a\}$, $w\in Y$, in her first move, then $d_{G-x_1}(w,u)=3$, since $w$ is not adjacent to $u$; if Alice colors a vertex $w \in V(G-x_1) \backslash \{v,x,x',x_2, \ldots ,x_a, y_1, \ldots ,y_a\}$, $w\in X$, in her first move, then $d_{G-x_1}(w,x) \geq 3$; if Alice colors $x$ or $x'$ in her first move, then $d_{G-x_1}(x,x')=3$; if Alice colors $x_i$ or $y_i$, $i \in \{2, \ldots ,a\}$, in her first move, then $d_{G-x_1}(x_i,y_i)=3$; if Alice colors $y_1$ in her first move, then $d_{G-x_1}(y_1,u)=3$, since $y_1$ is not adjacent to $u$; if Alice colors $v$ in her first move, then $d_{G-x_1}(v,x'')=3$. Considering all possibilities, and using Lemma~\ref{distance}, at least three different colors are needed to finish the game on $G-x_1$, which is a contradiction. We have thus seen that $Y=\{y_1, \ldots ,y_a\}$. If $V(G) \backslash \{x,x',x_1, \ldots ,x_a,y_1, \ldots ,y_a\} \neq \emptyset$, then there exists a vertex $u \in V(G) \backslash \{x,x',x_1, \ldots ,x_a,y_1, \ldots ,y_a\}$ that also belongs to $X$ and since $Y=\{y_1, \ldots ,y_a\}$, $G-u$ must be connected. Suppose that Alice and Bob play an A- or AB-independence coloring game on the graph $G-u$. If Alice colors a vertex $w \in V(G-u) \backslash \{x,x',x_1, \ldots ,x_a,y_1, \ldots ,y_a\}$, $w\in X$, in her first move, then $d_{G-u}(w,x) \geq 3$; if Alice colors $x$ or $x'$ in her first move, then $d_{G-u}(x,x')=3$; if Alice colors $x_i$ or $y_i$, $i \in \{1, \ldots ,a\}$, in her first move, then $d_{G-u}(x_i,y_i)=3$. Considering all possibilities, and using Lemma~\ref{distance}, at least three different colors are needed to finish the game on $G-u$, which is a contradiction. What remains to consider is the case $V(G)=\{x,x',x_1, \ldots ,x_a,y_1, \ldots ,y_a\}$. If there exists indices $j,h \in \{1, \ldots ,a\}$, $j \neq h$, such that $x_hy_j \notin E(G)$, then we let Alice and Bob play an A- or AB-independence coloring game on the graph $G-x_j$. If Alice colors $x$ or $x'$ in her first move, then $d_{G-x_j}(x,x')=3$; if Alice colors $x_i$ or $y_i$, $i \in \{1, \ldots ,a\} \backslash \{j\}$, in her first move, then $d_{G-x_j}(x_i,y_i)=3$; if Alice colors $y_j$ in her first move, then $d_{G-x_j}(x_h,y_j)=3$. Considering all possibilities, and using Lemma~\ref{distance}, at least three different colors are needed to finish the game on $G-x_j$, which is a contradiction. Finally, we have that $x_hy_j \in E(G)$ for every two indices $j,h \in \{1, \ldots ,a\}$, $j \neq h$, and consequently all $y_i$'s are distinct (we have $a$ of them). Obeying all adjacencies in $G$ and $V(G)=\{x,x',x_1, \ldots ,x_a,y_1, \ldots ,y_a\}$, we get $G=K_{a+1,a+1} - \mathcal{M}$, where $\mathcal{M}$ is a perfect matching in $G$. We know that $\cigA(G)=\cigAB(G)=a+1$~\cite{db-2021}, and since we have $\cigA(G)=\cigAB(G)=3$, it follows that $a=2$. Thus, $G=K_{3,3} - \mathcal{M}$, which is one of the $3$-$\cigA$-lower-game-vertex-critical ($3$-$\cigAB$-lower-game-vertex-critical) graphs.
\qed
\end{proof}

%%%%%%%%%%%%%%%%%%%%%%%%%%%%%%%%%%%%%%%%%%%%%%%%%%%%%%%%%%%%%%%%%%%%%
\section{Concluding remarks}
%%%%%%%%%%%%%%%%%%%%%%%%%%%%%%%%%%%%%%%%%%%%%%%%%%%%%%%%%%%%%%%%%%%%%

Despite many results presented in this paper, there is still a lot to research regarding game chromatic vertex-criticality. For instance, we did not give an answer to the question of existence of a $\chi_i$-game-vertex-critical graph $G$ such that the difference $\chi_i(G)-\chi_i(G-x)$, $x \in V(G)$, is arbitrarily large. However, we did prove that this difference is $1$ for $k$-$\chi_i$-lower-game-vertex-critical graphs when $k \in \{2,3,4\}$, which might lead us to think that this is true for all $k$-$\chi_i$-lower-game-vertex-critical graphs. Since we were not able to find any $\chi_i$-mixed-game-vertex-critical graphs, some weird behaviour might occur with them (if they even exist). To be more precise, the only $\varphi$-mixed-vertex-critical graphs we were able to find was for $\varphi=\cigAB$. Our first concluding thoughts lead us to propose the first two problems.

\begin{problem}
Does there exists a $\chi_i$-game-vertex-critical graph $G$ such that the difference $\chi_i(G)-\chi_i(G-x)$, $x \in V(G)$, is arbitrarily large?
\end{problem}

\begin{problem}
Find a $\varphi$-mixed-game-vertex-critical graph for $\varphi \in \{\chi_g,\chi_i, \cigA\}$.
\end{problem}

As for the $\chi_i$-upper-game-vertex-critical graphs, we constructed a graph in which the removal of a specific vertex causes the indicated chromatic number to increase, but unfortunately this graph is not vertex-critical. The constructed graph $G$ is shown in Figure~\ref{fig:twisted-graph}. It is easy to see, that $\chi_i(G)=3$. Namely, if Ann indicates vertices in the order $f,\ e,\ g,\ h,\ x,\ d,\ c,\ b,\ a$, Bob always has only one available color to properly color every selected vertex, and hence, she wins the indicated coloring game on $G$ using three colors. Moreover, the graph $G-x$ is well known under the name ``the twisted diamond'' and in \cite{gr-2012} it was proven that $\chi_i(G-x)=4$. For this reason we suspect that the $\chi_i$-upper-game-vertex-critical graphs do exist.

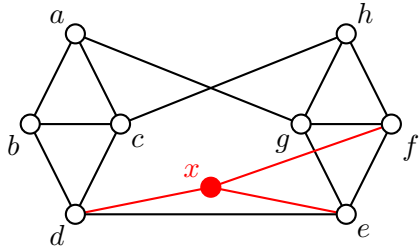
\begin{figure}[!ht]
\centering
\begin{tikzpicture}[scale=1.2, style=thick]
\def\vr{3pt}
\def\vm{0.5pt}
\def\len{1}

\coordinate(x) at (0,-0.7);
\coordinate(a) at (-1.5,1);
\coordinate(b) at (-2,0);
\coordinate(c) at (-1,0);
\coordinate(d) at (-1.5,-1);
\coordinate(e) at (1.5,-1);
\coordinate(f) at (2,0);
\coordinate(g) at (1,0);
\coordinate(h) at (1.5,1);

\draw[anchor = south east, red] (x) node {$x$};
\draw[anchor = south east] (a) node {$a$};
\draw[anchor =  north east] (b) node {$b$};
\draw[anchor = north west] (c) node {$c$};
\draw[anchor = north east] (d) node {$d$};
\draw[anchor = north west] (e) node {$e$};
\draw[anchor = north west] (f) node {$f$};
\draw[anchor =  north east] (g) node {$g$};
\draw[anchor = south west] (h) node {$h$};

\draw (a) -- (b) -- (c) -- (a);
\draw (h) -- (g) -- (f) -- (h);
\draw (a) -- (g);
\draw (c) -- (h);
\draw (b) -- (d);
\draw (c) -- (d);
\draw (d) -- (e);
\draw (g) -- (e);
\draw (f) -- (e);

\draw[red] (f) -- (x);
\draw[red] (d) -- (x);
\draw[red] (e) -- (x);

\draw[red](x)[fill=red] circle(\vr);
\draw(a)[fill=white] circle(\vr);
\draw(b)[fill=white] circle(\vr);
\draw(c)[fill=white] circle(\vr);
\draw(d)[fill=white] circle(\vr);
\draw(e)[fill=white] circle(\vr);
\draw(f)[fill=white] circle(\vr);
\draw(g)[fill=white] circle(\vr);
\draw(h)[fill=white] circle(\vr);
\end{tikzpicture}
\caption{Graph $G$ with $\chi_i(G)=3$ and $\chi_i(G-x)=4$}
\label{fig:twisted-graph}
\end{figure}

We found all $k$-$\varphi$-vertex-critical graphs for $k=2$, but solving this problem for $k \geq 3$ becomes considerably more challenging. Nevertheless, we were able to give a characterization of (connected) $3$-$\varphi$-lower-vertex-critical graphs, $\varphi \in \{\chi_g,\chi_i, \cigA, \cigAB\}$, with the crown jewel being Theorem~\ref{3-indep-lower}, which characterizes connected $3$-($\cigA$,$\cigAB$)-lower-vertex-critical graphs. However, this results still relies on the assumption ``lower''. The removal of the extra assumption leads us to our final problem.

\begin{problem}
Characterize (connected) $3$-$\varphi$-game-vertex-critical graphs for $\varphi \in \{\chi_g,\chi_i,\cigA,\cigAB\}$.
\end{problem}

%%%%%%%%%%%%%%%%%%%%%%%%%%%%%%%%%%%%%%%%%%%%%%%%%%%%%%%%%%%%%%%%%%%%%
\section*{Acknowledgements}
%%%%%%%%%%%%%%%%%%%%%%%%%%%%%%%%%%%%%%%%%%%%%%%%%%%%%%%%%%%%%%%%%%%%%

The authors acknowledge the financial support from the Slovenian Research Agency (research core funding No.\ P1-0297 and research projects J1-9109, J1-1693 and N1-0095).


\begin{thebibliography}{99}

\bibitem{an-2009} S.D.~Andres, The incidence game chromatic number, Discrete Appl.\ Math.\ 157 (2009) 1980--1987. 

\bibitem{bgk-08} T.~Bartnicki, J.~Grytczuk, H.A.~Kierstead, The game of
arboricity, Discrete Math.\ 308 (2008) 1388--1393.

\bibitem{bagr-07} T.~Bartnicki, J.~Grytczuk, H.~A.~Kierstead, X.~Zhu, The
map coloring game, Amer.\ Math.\ Monthly 144 (2007) 793--803.

\bibitem{bo-1991} H.L.~Bodlaender, On the complexity of some coloring games, Internat.\ J.\ Found.\ Comput.\ Sci.\ 2 (1991) 133--147.

\bibitem{bgj-2019} B.~Bosek, J.~Grytczuk, G.~Jak\'{o}bczak, Majority coloring game, Discrete Appl.\ Math.\ 255 (2019), 15--20.

\bibitem{bmd-2021} B.~Bre{\v{s}}ar, M.~Jakovac, D.~{\v{S}}tesl,
Indicated coloring game on Cartesian products of graphs,
Discrete Appl.\ Math.\ 289 (2021) 320--326.

\bibitem{brklra-2010}
  B.~Bre{\v{s}}ar, S.~Klav{\v{z}}ar, D.~F.~Rall,
  Domination game and an imagination strategy,
  SIAM J. Discrete Math. 24 (2010) 979--991.

\bibitem{db-2021} B.~Bre{\v{s}}ar, D.~{\v{S}}tesl, The independence coloring game on graphs, (2021) https://arxiv.org/abs/2103.13656.

\bibitem{dgc-2015} Cs.~Bujt\'{a}s, S.~Klav{\v{z}}ar, G.~Ko{\v{s}}mrlj, Domination game critical graphs, Discuss. Math. Graph Theory\ 35 (2015) 781--796.

\bibitem{cs-2013} C.~Charpentier, \'{E}.~Sopena, 
Incidence coloring game and arboricity of graphs, 
Lecture Notes Comp.\ Sci.\ 8288 (2013) 106--114.

\bibitem{dizh-99} T.~Dinski, X.~Zhu, Game chromatic number of graphs,
Discrete Math.\ 196 (1999) 109--115.

\bibitem{faigle} U. Faigle, U. Kern, H.~A. Kierstead, W.~T.~Trotter, On the game chromatic number of some classes of graphs, Ars Combin.\ 35 (1993) 143--150.

\bibitem{F} A.S.~Fraenkel, Combinatorial games: selected bibliography with a succinct gourmet introduction, Electron. J. Combin., Dynamic Survey 2 (2007) 78pp; \texttt{https://www.combinatorics.org/ojs/index.php/eljc/article/view/DS2}.

\bibitem{ga-81}
  M.~Gardner,
  Mathematical games,
  Scientific American 244 (1981) 18--26.
  
  \bibitem{go-he-2018}
W. Goddard, M.~A. Henning, The competition-independence game in trees, J.\ Combin.\ Math.\ Combin.\ Comput.\ 104 (2018) 161--170.
  
\bibitem{gr-2012} A.~Grzesik, Indicated coloring of graphs, 
  Discrete Math.\ 312 (2012) 3467--3472.

\bibitem{tdgc-2018} M.A.~Henning, S.~Klav{\v{z}}ar, D.F.~Rall,
Game total domination critical graphs, Discrete Appl.\ Math.\ 250 (2018) 28--37.
  
\bibitem{kiko-2009} H.A.~Kierstead, A.~Kostochka, Efficient graph packing via game coloring, Combin.\ Probab.\ Comput.\ 18 (2009) 765-774.

\bibitem{kitr-01}
  H.A.~Kierstead, T.~Trotter,
  Competitive colorings of oriented graphs,
  Electron. J. Combin. 8 (2001) \#R12, 15pp.
  
\bibitem{kir-2012}
 H.A.~Kierstead, C.-Y.~Yang, D.~Yang, X.~Zhu,
Adapted game colouring of graphs, European J.\ Combin.\ 33 (2012) 435--445.
  
\bibitem{la-2014} M.~Laso\'{n}, Indicated coloring of matroids, Discrete Appl.\ Math.\ 179 (2014) 241--243.

\bibitem{neso-01} J.~Ne\v set\v ril, \'{E}.~Sopena, On the oriented game chromatic number, Electron.\ J.\ Combin.\ 8 (2001) R\#14, 13pp.

\bibitem{raj-2015} R.P.~Raj, S.F.~Raj, H.P.~Patil, On indicated coloring of graphs, Graphs Combin.\ 31 (2015) 2357--2367.

\bibitem{raj-2017} S.F.~Raj, R.P.~Raj, H.P.~Patil, On indicated chromatic number of graphs, Graphs Combin.\ 33 (2017) 203--219.

\bibitem{tz-2015} Zs.~Tuza, X.~Zhu, Colouring games, in [Topics in chromatic graph theory, {\it Cambridge Univ.\ Press}, Cambridge, 2016], 304--326. 

\bibitem{book} D.B.~West, Introduction to Graph Theory, Prentice Hall, Inc., Upper Saddle River, 360 NJ, 1996.

\end{thebibliography}
\end{document}